\documentclass[letterpaper, 11pt,oneside,reqno]{amsart}
\usepackage[utf8]{inputenc}
\usepackage[T1]{fontenc}
\usepackage{baskervald}
\usepackage[letterpaper]{geometry}
\usepackage[backref,linkcolor=black,colorlinks=true]{hyperref}
\usepackage{microtype}

\usepackage[square,sort,comma,numbers]{natbib}
\usepackage[draft]{fixme}
\fxsetup{theme=color, mode=multiuser, margin, noinline}
\FXRegisterAuthor{td}{TD}{\color{cyan}Tim}
\FXRegisterAuthor{sa}{SA}{\color{red}Sameer}
\FXRegisterAuthor{rt}{RT}{\color{blue}Rekha}
\FXRegisterAuthor{ml}{ML}{\color{orange}Max}

\usepackage{booktabs,mathtools, graphicx, amssymb, breqn, xcolor,tikz-cd,  ulem,mathrsfs,thmtools, thm-restate}
\usepackage[noabbrev]{cleveref}
\usetikzlibrary{decorations.text,calc,arrows.meta}
\DeclareMathOperator{\rank}{rank}
\DeclareMathOperator{\diag}{diag}
\DeclareMathOperator{\adj}{adj}

\DeclareMathOperator{\PGL}{\rm PGL}
\DeclareMathOperator{\minors}{minors}
\DeclareMathOperator{\im}{im}
\DeclareMathOperator{\lcm}{lcm}

\DeclareMathOperator{\cl}{cl}

\newcommand{\GL}{\textup{GL}}
\newcommand{\PP}{\mathbb{P}}
\newcommand{\CC}{\mathbb{C}}
\newcommand{\RR}{\mathbb{R}}
\newcommand{\ZZ}{\mathbb{Z}}
\newcommand{\kfour}{\text{2-, 3-, and  4}}
\newcommand{\kthree}{\text{2- and 3}}

\newcommand{\V}{\mathrm{V}}
\newcommand{\Vaff}{\mathrm{V}_a}
\DeclareMathOperator{\ideal}{I}
\newcommand{\iideal}[1]{{I}_{#1}}
\newcommand{\row}{\mathbf{r}}

\newcommand{\kfoc}{\left(\bA \, | \, \pp \right) [\row ]_\sigma}
\newcommand{\BM}{\begin{matrix}}
\newcommand{\EM}{\end{matrix}}
\newcommand{\si}{\sigma}
\newcommand{\pp}{\mathbf{p}}

\newcommand{\bA}{\mathbf{A}}
\newcommand{\qq}{\mathbf{q}}

\newcommand{\xx}{\mathbf{x}}
\newcommand{\yy}{\mathbf{y}}

\newcommand{\Aqp}[1][{}]{\iideal{\bA , \qq_{#1} , \pp_{#1}}}
\newcommand{\Ap}[1][{}]{\iideal{\bA , \pp_{#1}}}
\newcommand{\Abarqp}[1][{}]{\iideal{\bar{\bA} , \qq_{#1} , \pp_{#1}}}
\newcommand{\Abarp}[1][{}]{\iideal{\bar{\bA} , \pp_{#1}}}
\newcommand{\GammaAqp}{\Gamma_{\bA, \qq , \pp}}
\newcommand{\GammaAp}{\Gamma_{\bA, \pp}}
\newcommand{\GammaAbarqp}{\Gamma_{\bar{\bA}, \qq , \pp}}
\newcommand{\GammaAbarp}{\Gamma_{\bar{\bA} , \pp}}
\newcommand{\minorideal}{{M^{m,1}_{\bA , \qq , \pp}}}

\newcommand{\suchthat}{\text{ s.t.  }}
\newcommand{\stacked}[2]{\left( \begin{array}{c|c|c}\hspace{-.4em}{#1}^\top&\cdots&{#2}^\top\end{array}\hspace{-.4em}\right)}

\newtheorem{lemma}{Lemma}[section]
\newtheorem{theorem}[lemma]{Theorem}
\newtheorem{proposition}[lemma]{Proposition}
\newtheorem{definition}[lemma]{Definition}
\newtheorem{corollary}[lemma]{Corollary}

\newtheorem{remark}[lemma]{Remark}
\theoremstyle{definition}
\newtheorem{examplex}{Example}[section]
\newenvironment{example}
  {\pushQED{\qed}\examplex}
  {\popQED\endexamplex}
\numberwithin{equation}{section}

\newcommand{\frakm}{\mathfrak{m}}

\newcommand{\hide}[1]{}

\usepackage{mathrsfs}

\DeclareMathOperator{\Hilb}{Hilb}

\title{An Atlas for the Pinhole Camera}
\author{Sameer Agarwal}
\address{Google Inc., Seattle}
\email{sameeragarwal@google.com}
\author{Timothy Duff}
\address{Department of Mathematics, University of Washington, Seattle}
\email{timduff@uw.edu}
\author{Max Lieblich}
\address{Department of Mathematics, University of Washington, Seattle}
\email{lieblich@uw.edu}
\author{Rekha R. Thomas}
\address{Department of Mathematics, University of Washington, Seattle}
\email{rrthomas@uw.edu}

\begin{document}
\maketitle
\begin{abstract}
We introduce an atlas of algebro-geometric objects associated with image formation in pinhole cameras. The nodes of the atlas are algebraic varieties or their vanishing ideals related to each other by projection or elimination and restriction or specialization respectively. This atlas offers a unifying framework for the study 
of problems in 3D computer vision.  We initiate the study of the atlas by completely characterizing a part of the atlas stemming from the triangulation problem.
We conclude with several open problems and generalizations of the atlas.
\end{abstract}

\section{Introduction}
The standard model of a pinhole camera (also known as a \emph{projective camera}) in computer vision is a surjective linear projection
$\PP^3\dashrightarrow\PP^2.$
Such a map is specified by a $3\times 4$-matrix $A$ of rank $3$ up to scaling, which can be realized as a point in $\PP^{11}$. 
This map sends a \emph{world point} $q$ in $\PP^3$, which can be realized as a $4$-vector up to scale, 
to its image point $p \in\PP^2$, which can be realized as a $3$-vector up to scale. 
The center of the projection map, also called the \emph{camera center}, can be identified with the kernel of the representing matrix $A$. If we wish to express that a point $q$ is sent to 
the image $p$ by the camera $A$, we will write $A q\sim p$ to indicate equality up to scale.

Much work has been done on the algebraic varieties obtained from a {\em fixed} arrangement of cameras $\bar{\bf{A}} = (\bar{A}_1,\ldots,\bar{A}_m)$\footnote{To distinguish between known and unknown quantities, we use a bar over
an object to indicate specialization. For instance, $A$ stands for a symbolic $3 \times 4$ matrix denoting a camera while $\bar{A}$ 
is a $3 \times 4$ scalar matrix realizing a camera. We also use bold face letters to indicate collections.
For instance, we use $\bA$ and $\bar{\bA}$ to specify a collection of symbolic and scalar cameras respectively. }
by taking the closed image of the rational imaging map
\begin{align}\label{eq:imaging map0}
\begin{split}
\varphi_{\bar{\bA}} : \PP^3 &\dashrightarrow (\PP^2)^m \\
q &\mapsto (\bar{A}_1 q , \ldots , \bar{A}_m q ).
\end{split}
\end{align}

Such a variety is known as a \emph{multiview variety}~\cite{APT21, AST13}. Several closely-related sets have been studied in computer vision: for example, the names {\em joint image} and {\em natural descriptor} for the constructible set $\im \varphi_{\bar{\bA}}$ were coined, respectively, by Triggs~\cite{Tri95} and Heyden and \AA str\"{o}m~\cite{HA97}. Much previous work has also gone into studying various polynomials that vanish on the multiview variety, eg.~\cite{DBLP:conf/iccv/FaugerasM95,HA97,ma2004invitation,DBLP:conf/iccv/TragerHP15,Tri95}. The associated \emph{multiview ideals} have been calculated in~\cite{APT21,AST13}, and their moduli have been considered in \cite{AST13,LV20}.

Consider now the universal incarnation of this problem where we treat the cameras and the world points as unknowns. In particular, one can think about the universal imaging map
\begin{equation}\label{eq:imagining-map-m-1-n-1}
\begin{split}
    \PP^{11} \times \PP^3 &\dashrightarrow \PP^2\\
    (A , q) &\mapsto A q
\end{split}
  \end{equation}  
that sends a pair $(A,q)$ to $A q$. This map describes the image of one unknown point by one unknown camera. More generally, fixing integers $m,n\ge 1$ there is a map
\begin{align} \label{eq:imaging-map}
(\PP^{11})^m\times(\PP^3)^n\dashrightarrow(\PP^2)^{mn}
\end{align}
which, where defined, associates $m$ unknown cameras $A_1, \ldots, A_m \in \PP^{11}$ and $n$ 
unknown world points $q_1, \ldots, q_n\in \PP^3$ to their $mn$ images $A_i q_j \sim p_{ij} \in \PP^2.$\footnote{
The imaging map of~\eqref{eq:imaging-map} models a scenario in which all points are visible in every image.
One may consider other scenarios of interest in computer vision, eg.~when each point is visible in only some of the images (eg.~\cite{calTrif,PL1P}).
}

For any nonempty Zariski-open $U\subset (\PP^{11})^m\times(\PP^3)^n$ where the imaging map 
\eqref{eq:imaging-map} is defined, we can consider the Zariski closure of the graph of the imaging map restricted\footnote{We note that the definition of $\GammaAqp^{m,n}$ is independent of the choice of $U$.
As such, it is insensitive to certain physical assumptions about the camera matrices (eg.~ that they have full rank, or that their centers do not coincide.)
In particular, although a \emph{generic} point $(\bar{\bA}, \bar{\qq}, \bar{\pp}) \in \GammaAqp^{m,n}$ will be such that each $A_i$ has full rank and all $A_i q_j $ are defined, these conditions do not hold for an \emph{arbitrary} point $(\bar{\bA}, \bar{\qq}, \bar{\pp}) \in \GammaAqp^{m,n}.$
} to $U$:
\begin{equation}\label{eq:GammaAqp}
\GammaAqp^{m,n} = \operatorname{cl}\left\{ \left( (\bar{\bA}, \bar{\qq}),  \bar{\pp}\right) \in U \times (\PP^2)^{mn} \mid \bar{A}_i \bar{q}_j \sim \bar{p}_{i j} \, \forall i = 1, \ldots, m, \, j =1 , \ldots, n \right\}.
\end{equation}

The set of polynomials vanishing on all points of $\GammaAqp^{m,n}$ can be understood as the set of all constraints that must be satisfied for any valid geometry $(\bar{\bA}, \bar{\qq}, \bar{\pp})\in \GammaAqp^{m,n}.$
This set forms an ideal in the polynomial ring $\CC [\bA, \qq , \pp],$ 
generated by polynomials which are homogeneous in each of the $m+n+mn$ groups of variables $A_1, \ldots, A_m$, $q_1, \ldots ,
q_n$, $p_{1 1} , \ldots , p_{m n}$. 

\begin{definition}
\label{def:mother-ideal}
The \textbf{image formation correspondence} $\GammaAqp^{m,n}$ is the algebraic variety defined by~\eqref{eq:GammaAqp}.
We denote its vanishing ideal (\Cref{def:vanishing ideal}) by $\Aqp^{m,n}.$
\end{definition}

In this paper, we will be concerned with the structure of $\Aqp^{m,n}$ and the vanishing ideals of varieties which may be defined in terms of three natural geometric operations on  $\GammaAqp^{m,n}$.
\begin{enumerate}
\item \emph{Coordinate projection}. For example, let $\pi_{\qq} : \GammaAqp^{m,n} \to (\PP^{11})^m \times (\PP^2)^{mn}$ denote the coordinate projection $\pi_\qq (\bar{\bA}, \bar{\qq}, \bar{\pp}) = 
(\bar{\bA}, \bar{\pp}).$
We denote its image by
\begin{equation}\label{eq:GammaAp}
\GammaAp^{m,n} := \pi_{\qq} (\GammaAqp^{m,n}) = 
\{ (\bar{\bA}, \bar{\pp}) \mid \exists \bar{\qq} \text{ s.t. } (\bar{\bA}, \bar{\qq}, \bar{\pp}) \in \GammaAqp^{m,n} \} .
\end{equation}
We note that, by the projective elimination theorem (see eg.~\cite[Theorem 4.22]{MS21}), 
$\GammaAp^{m,n}$ is closed in the Zariski topology on $(\PP^{11})^m \times (\PP^{2})^{mn}.$  
\item \emph{Coordinate specialization}. For example, let $\bar{\bA}$ be a particular arrangement of cameras. 
We can form the intersection of $\GammaAqp^{m,n}$ with the coordinate planes where $A_i = \bar{A}_i$ for $i=1,\ldots , m$, and then coordinate project away from $(\PP^{11})^m$ to obtain
\begin{equation}\label{eq:GammaAbarqp}
\GammaAbarqp^{m,n} := 
\{ (\bar{\qq}, \bar{\pp} )  \mid (\bar{\bA}, \bar{\qq}, \bar{\pp}) \in \GammaAqp^{m,n} \}
.
\end{equation}
\item \emph{Projection + specialization}. Combining 1 and 2, the relevant variety is
\begin{equation}\label{eq:GammaAbarp}
\GammaAbarp^{m,n} = \{ \bar{\pp}  \mid \exists \bar{\qq} \text{ s.t. } (\bar{\bA}, \bar{\qq}, \bar{\pp}) \in \GammaAqp^{m,n} \}.
\end{equation}
\end{enumerate}
\begin{figure}[t]
\begin{center}
\begin{tikzcd}[]
& & 
{\Gamma_{\bar{\bA}, \qq,\pp}^{m,n} 
\arrow[dl] 
\arrow[from=dr, color=red]
} 
& 
{
\Gamma_{\bA, \qq, \bar{\pp}}^{m,n} 
}
& 
{\Gamma_{\bA, \bar{\qq},\pp}^{m,n} 
\arrow[dr]  
}
& & 
\\
&
{
\Gamma_{\bar{\bA}, \pp}^{m,n}
} 
& 
\Gamma_{\bA, \bar{\pp}}^{m,n} \arrow[from=ur, crossing over]
& 
{
\Gamma_{\bA, \qq, \pp}^{m,n} 
\arrow[dr] 
\arrow[dl]
\arrow[u,color=red, crossing over]
\arrow[ur,color=red, crossing over]
}
&
{\Gamma_{\qq,\bar{\pp}}^{m,n}}\arrow[from=ul, crossing over]
& 
{\Gamma_{\bar{\qq}, \pp}^{m,n}}
&
\\ 
&
& 
{
\Gamma_{\bA, \pp}^{m,n}
\arrow[ul,color=red]
\arrow[dr] 
\arrow[u,color=red]
} 
& 
& 
{
\Gamma_{\qq, \pp}^{m,n}
\arrow[ur,color=red]
\arrow[dl] 
\arrow[u,color=red]
} 
& 
&
\\ 
&
& 
& 
{
\Gamma_{ \pp}^{m,n}
} 
& 
&
&
\end{tikzcd}
\end{center}
\caption{{\bf Atlas of the Pinhole Camera.} Red arrows indicate variable specialization and black arrows indicate projection. An analogous diagram can be constructed where the arrows remain the same but the varieties are replaced by their vanishing ideals.}
\label{fig:full-diagram}
\end{figure}
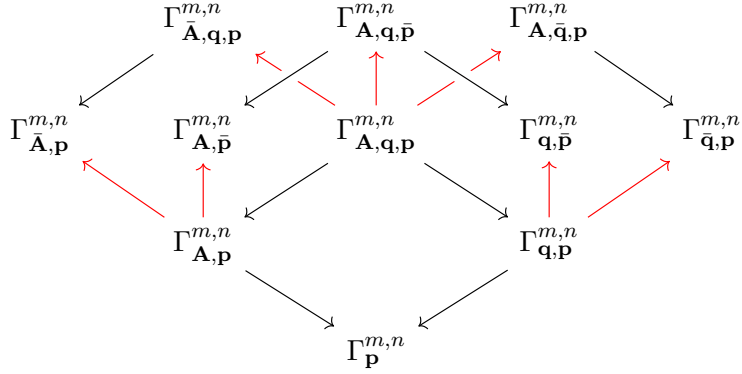

All the non-trivial varieties (and their vanishing ideals) obtained in this way can be organized as shown in~\Cref{fig:full-diagram}\footnote{Not every result of successively applying projection and specialization operations to $\Gamma_{\bA, \qq, \pp}^{m,n}$ is included here.  For example $\Gamma^{m,n}_{\bA}$ and $\Gamma^{m,n}_{\qq}$ are trivial. Similarly $\Gamma^{m,n}_{\bar{\bA},\qq,\bar{\pp}}$ and $\Gamma^{m,n}_{\bA,\bar{\qq},\bar{\pp}}$ are defined by linear equations and not interesting for projective cameras. However this can change as the model for the camera is varied. For example, $\Gamma^{m,n}_{\bA,\bar{\qq},\bar{\pp}}$ is an interesting nonlinear variety for Euclidean cameras. See~\Cref{sec:open-problems} for more. }. Red arrows indicate variable specialization and black arrows indicate variable elimination.

The aim of this paper is to motivate and initiate the study of this diagram which we will call the {\em Atlas of the Pinhole Camera}. We call it an atlas because it is a systematic collection of algebro-geometric objects associated with the pinhole camera that also captures the relationships between them, telling us how one can travel from one object to another. 
For example, knowing that a variety is obtained by specializing a group of variables (red arrow) can be used to easily compute its dimension as the difference of the dimensions of two varieties. 
See \Cref{sec:dimension-counts} for dimension counts for each of the varieties in the atlas.

\subsection{Computer Vision and the Atlas}\label{sec:estimation}
As the ideal $\Aqp^{m,n}$ describes {\em all} the algebraic relationships that hold between cameras, world, and image points in 3D reconstruction, many of the problems studied in 3D computer vision (also known as multiview geometry) can be described in terms of the varieties and ideals that occur in the atlas. So we begin our study of the atlas by summarizing what is known about each node and how it relates to problems in multiview geometry. 

The image formation correspondence, $\Gamma_{\bA, \qq, \pp}^{m,n}$ is the starting point of our study. 
Given noisy image observations $\widetilde{\pp}$, the problem of finding $\bA$ and $\qq$ such that they best explain these image observations is the {\em reconstruction problem}. Other names for this problem are {\em Structure from Motion} (SfM)~\cite{HZ04} and {\em Simultaneous Localization And Mapping} (SLAM)~\cite{thrun2005probabilistic}.

Assuming Gaussian noise, the maximum likelihood estimate of $\bA$ and $\qq$ can be found by solving the following optimization problem:
\begin{align}
		\displaystyle\arg \hspace{-.4em} \min_{(\bar{\bA},\bar{\qq}, \bar{\pp})} \sum_{i}^n\sum^{m}_{j}\|\bar{p}_{ij} - \widetilde{p}_{ij} \|^2 \suchthat (\bar{\bA},\bar{\qq}, \bar{\pp}) \in \GammaAqp^{m,n}. \label{eq:bundle-adjustment}
\end{align}
The quantity $\|\bar{p}_{ij} - \widetilde{p}_{ij}\|^2$, known as the {\em reprojection error}, is understood to be the squared Euclidean norm between dehomogenized $p$ and $\widetilde{p}$. That is, we assume that the $\bar{p}$ and $\widetilde{p}$ are finite points. 
The optimization problem~\eqref{eq:bundle-adjustment} is also known as the {\em bundle adjustment} problem~\cite{triggs1999bundle}. {\em Projective factorization} is another method for solving the reconstruction problem~~\cite{nasihatkon2015generalized,triggs1996factorization}. In this method, 
one tries to find a nearby feasible point on $\Gamma^{m,n}_{\bA,\qq, \pp}$ by solving a matrix factorization problem.

Eliminating $\mathbf{q}$ (also known as {\em structure}) from the constraint set in~\eqref{eq:bundle-adjustment} we obtain $\GammaAp^{m,n}$ and the {\em bundle adjustment without structure} problem:
\begin{align}
    \arg\min_{\bar{\bA}, \bar{\pp}} \sum_{i}^m\sum_{j}^n\|\bar{p}_{ij} - \widetilde{p}_{ij}\|^2 \text{ s.t. } (\bar{\bA}, \bar{\pp}) \in \GammaAp^{m,n}.
    \label{eq:bundle-adjustment-without-structure}
\end{align}
Bundle adjustment without
structure is of interest because eliminating the
world points $\mathbf{q}$ significantly reduces the dimensionality of the problem~\cite{persson2019global,rodriguez2011reduced,rupnik2020towards,isprs-annals-IV-2-W3-81-2017,steffen2010relative}. The two-view $(m=2)$ version of this problem is particularly important, as it is usually the first step in incremental 3D reconstruction algorithms.

Similarly $\Gamma_{\qq, \pp}^{m,n}$ is obtained from $\GammaAqp^{m,n}$ by eliminating the cameras (also known as {\em motion}) and $\Gamma_{\pp}^{m,n}$ is obtained from  $\GammaAqp^{m,n}$ by eliminating both cameras and points.
One can formulate the reconstruction problem as a nearness problem on each of these varieties, and each formulation leads to its own version of the bundle adjustment problem. As a result we will refer to all four ideals $\Ap^{m,n}, \Aqp^{m,n}, \iideal{\qq , \pp}^{m,n}, \iideal{\pp}^{m,n}$ as {\em bundle adjustment ideals}.

Of these four ideals, only $\Ap^{m,n}$ has been studied to a certain extent~\cite{DBLP:conf/iccv/TragerHP15}. 
Before our work, very little was known about the structure of the other three ideals. This may come as a surprise to the reader given the importance of the reconstruction problem in 3D computer vision. 
The lack of study can be explained by computational reasons. Substituting $p_{ij} \sim A_i q_j$ into the objective and dehomogenizing it gives us the more commonly occurring form of~\eqref{eq:bundle-adjustment} as an unconstrained rational optimization problem. This form is preferred in practice because it can be solved much more efficiently using simpler algorithms than the form in~\eqref{eq:bundle-adjustment}~\cite{triggs1999bundle}. 
Similarly, the bundle adjustment without structure problem is usually solved by approximating~\eqref{eq:bundle-adjustment-without-structure} by an unconstrained optimization problem where the objective is the sum of squares of certain polynomials (the so-called epipolar and trifocal constraints~\cite{HZ04}).

These transformations have the disadvantage that they obscure the geometric structure of the problem. For example, by formulating the reconstruction problem as an optimization problem over the variety $\Gamma^{m,n}_{\bA,\qq, \pp}$, we can study it using the tools of complex and real algebraic geometry. It also reveals two other versions of the problem not considered before. This perspective has been useful in constructing effective algorithms for solving the {\em triangulation problem}.

In {\em triangulation}~\cite{HARTLEY1997146}, the cameras $\mathbf{A} =\bar{\bA}$ are known and fixed,  we are given the noisy images $\widetilde{\pp}$ of an unknown world point, and we seek a solution to 
	\begin{align}
		\arg\min_{\bar{\qq}, \bar{\pp}} \sum_i^m \|\bar{p}_i - \widetilde{p}_i\|^2 \text{ s.t. } (\bar{\bA}, \bar{\qq}, \bar{\pp}) \in \GammaAqp^{m,1}, 
  \label{eq:triangulation}
	\end{align}	 
 i.e., find the world point $q$ that best explains the image observations. The variety $\Gamma_{\bar{\bA}, \qq,\pp}^{m,n}$ is the slice of $\Gamma_{\bA, \qq, \pp}^{m,n}$ that is obtained by fixing the cameras. So the above optimization problem is better re-written in terms of $\Gamma_{\bar{\bA},\qq,\pp}^{m,1}$ as:
	\begin{align}
		\arg\min_{\bar{\qq}, \bar{\pp}} \sum_i^m \|\bar{p}_i - \widetilde{p}_i\|^2 \text{ s.t. } (\bar{\qq}, \bar{\pp}) \in \Gamma_{\bar{\bA},\qq,\pp}^{m,1}. 
  \label{eq:triangulation2}
	\end{align}	 
The {\em multiview variety} $\Gamma_{\bar{\bA},\pp}^{m,1}$ is obtained from $\Gamma_{\bar{\bA},\qq,\pp}^{m,1}$ by eliminating the world point from it. As we noted earlier, the multiview variety and sets related to it are perhaps the most well-studied objects in computer vision~\cite{AST13,HZ04,HA97,ma2004invitation,DBLP:conf/iccv/TragerHP15,Tri95}. Using $\Gamma_{\bar{\bA},\pp}^{m,1}$ we can re-write the triangulation problem as:
\begin{align}
		\arg\min_{\bar{\pp}} \sum_i^m \|\bar{p}_i - \widetilde{p}_i\|^2 \text{ s.t. }  \bar{\pp} \in \GammaAbarp^{m,1}. 
  \label{eq:structure-less-triangulation}
	\end{align}	 
This form of the triangulation problem has been used effectively in practice~\cite{AholtAgarwalThomas2012, HARTLEY1997146,lindstrom2010}. We will refer to  $\Abarqp^{m,n}$ and $\Abarp^{m,n}$ as the {\em triangulation ideals}.

The optimization problem in~\eqref{eq:structure-less-triangulation} finds the nearest point on the variety $\GammaAbarp^{m,1}$. One measure of complexity of problems of this form is the  concept of the {\em Euclidean distance degree} of a variety~\cite{draisma2016euclidean}. Recently the affine Euclidean distance degree of $\GammaAbarp^{m,1}$ was computed~\cite{maxim2020euclidean}, settling a conjecture of St\'ewenius et al~\cite{stewenius2005hard}. 

In resectioning~\cite[Chapter 7]{HZ04}, the world points $\qq = \bar{\qq}$ are known and fixed, we are given their noisy images $\widetilde{\pp}$ by an unknown camera $A$ and we seek a solution to:
	\begin{align}
		\arg\min_{\bar{\bA}, \bar{\pp}} \sum_j^n \|\bar{p}_j - \widetilde{p}_j\|^2 \text{ s.t. } (\bar{\bA}, \bar{\qq}, \bar{\pp}) \in \Gamma^{1,n}_{\bA, \bar{\qq},\pp}.
  \label{eq:resectioning}
	\end{align}
i.e., find the camera matrix $A$ that best explains the image observations.
Resectioning is also known as the {\em Perspective-$n$-Point Problem}. Like triangulation, it is a fundamental problem in 3D computer vision and a considerable effort has been devoted to solving it~\cite{lepetit2009epnp}.  The variety $\Gamma_{\bA, \bar{\qq},\pp}^{1,n}$ is the slice of $\Gamma_{\bA, \qq, \pp}^{1,n}$ that is obtained by fixing the world points, so like triangulation we can re-write~\eqref{eq:resectioning} as an optimization problem over $\Gamma_{\bA, \bar{\qq}, \pp}^{1,n}$. Also like triangulation, we can eliminate $\bA$ and formulate resectioning as a nearest point to the variety problem on $\Gamma_{\bar{\qq},\pp}^{1,n}$. As a result we will refer to $\ideal^{m,n}_{\bA,\bar{\qq},\pp}, \ideal^{m,n}_{\bar{\qq},\pp}$ as {\em resectioning ideals}.

The resectioning and triangulation problems are in a sense duals of each other. So one might be inclined to think that the structure of $\ideal_{\bar{\qq},\pp}^{m,n}$ would be as well-studied as that of $\ideal_{\bar{\bA},\pp}^{m,n}$. This however is not the case, very little is known about the structure of the resectioning ideals $\ideal^{m,n}_{\bA, \bar{\qq},\pp}$ and $\ideal_{\bar{\qq},\pp}^{m,n}$ and their varieties~\cite{368154,DBLP:conf/eccv/SchaffalitzkyZHT00}.

Let us now talk about the three varieties that are obtained by fixing the image points $\pp=\bar{\pp}$. From $\Gamma_{\bA, \pp}^{m,n}$ we obtain $\Gamma_{\bA, \bar{\pp}}^{m,n}$. 
The vanishing ideal of $\Gamma_{\bA, \bar{\pp}}^{m,n}$ is the set of polynomials on camera matrices that satisfy given image observations. 
Among these is a distinguished set of polynomials which are multilinear forms over the image coordinates. Longuet-Higgins was the first to study the matrix defining the bilinear form for a pair of Euclidean cameras -- the {\em essential matrix}~\cite{demazure:inria-00075672,hartley1995investigation,longuet1981computer}. The generalization of the essential matrix to projective cameras is known as the {\em fundamental matrix}~\cite{luong1996fundamental}. For three and four cameras we get the trifocal tensor~\cite{aholt2014ideal,hartley1997lines,HZ04} and the quadrifocal tensor~\cite{shashua2000structure, oeding2017quadrifocal} respectively. 
We note that the varieties of these tensors and their vanishing ideals have been studied~\cite{aholt2014ideal,oeding2017quadrifocal,demazure:inria-00075672}.
However, they do not appear in our atlas.

When we are given the minimal number of image observations such that $\Gamma^{m,n}_{\bA, \bar{\pp}}$ is zero dimensional (modulo $\PGL_4$), finding cameras $\bar{\bA}$ that lie on $\Gamma_{\bA, \bar{\pp}}^{m,n}$ is a fundamental problem in computer vision and the subject of intensive study. These problems are known as {\em minimal problems}.
Starting with Nist\'er's work~\cite{nister2004efficient} on estimating the essential matrix using five point correspondences, there has been an explosion of work on using methods of computational algebraic geometry for solving these problems~\cite{kukelova2011polynomial,larsson2017efficient,stewenius2005grobner}.

From $\Gamma_{\qq, \pp}^{m,n}$ we get $\Gamma_{\qq, \bar{\pp}}^{m,n}$, by fixing the image points $\pp=\bar{\pp}$. The vanishing ideal of this variety is the set of polynomials that a set of 3D points must satisfy to explain a set of given image observations. Finally, we have the intriguing variety $\Gamma_{\bA, \qq, \bar{\pp}}^{m,n}$. Points on this variety correspond to 3D reconstructions determined by a set of image observations. A particularly interesting case is the two view, seven point problem, which is closely related to the {\em seven point algorithm} for estimating the fundamental matrix~\cite{HZ04}. 

The above is a necessarily brief presentation. However, we hope that we have convinced the reader that the nodes/objects of the atlas are related to key problems in 3D computer vision. Some of these objects have been studied before, but many remain completely unexplored.
We hope that by systematically constructing these objects as slices and projections of a single variety, and organizing them in the atlas will reveal more of their shared structure and propel their study.

In the remainder of this paper, we initiate the study of the atlas by exploring a part of it in detail.  
The next section introduces our main results.

\subsection{A Square in the Atlas}
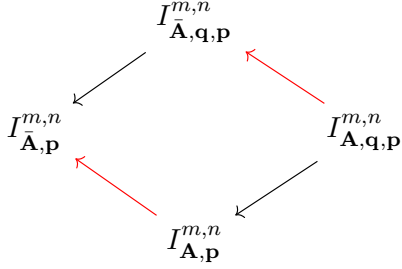
\begin{figure}[t]
\begin{center}
\begin{tikzcd}[]
&&
{
\Abarqp^{m,n}
\arrow[dl]
} 
& & 
\\
&
{
\Abarp^{m,n}
} 
&&
{
\Aqp^{m,n} 
\arrow[ul,color=red]
\arrow[dl]
}
&
\\ 
&
& 
{
\Ap^{m,n}
\arrow[ul, color=red]
} 
\end{tikzcd}
\end{center}
\caption{The four ideals studied in this paper.
}\label{fig:ideals}
 \end{figure}

Polynomial constraints, particularly those which are implied by rank constraints, pervade the multiview geometry literature.
An extensive catalogue of these constraints may be found, for instance, in~\cite[Chapter 8]{ma2004invitation}. 
In the context of the atlas, many of these constraints are sufficient to describe various varieties locally.
For instance, we have
\begin{equation}\label{eq:derive-rank-constraint}
Aq \sim p \quad \Rightarrow \quad \exists \lambda \suchthat Aq = \lambda p \quad \Rightarrow \quad \rank \begin{pmatrix} Aq &p \end{pmatrix} \le 1.
\end{equation}
Thus, the condition $Aq \sim p$ may be expressed as the vanishing of the $2\times 2$ minors of $\left( \BM Aq & p \EM \right),$ giving equations defining $\GammaAqp^{m,n}$ near a generic point.
However, the question of whether these equations \emph{globally} define the full closed correspondence $\GammaAqp^{m,n}$ is more subtle.
In fact, already for $2$ symbolic cameras $A_1,A_2$ and image points $p_1, p_2,$ the $2$-focal polynomial (cf.~\Cref{def:k-focal})
\[
\det \begin{pmatrix}
A_1 & p_1 & 0 \\
A_2 & 0   & p_2
\end{pmatrix} 
\]
is contained in $\Aqp^{2,1},$ and yet this polynomial is not in the ideal generated by all $2\times 2$ minors of the matrices $\left( \BM A_1 q & p_1 \EM \right)$ and $\left( \BM A_2 q & p_2 \EM \right).$ 
Thus, in a precise sense, the rank constraints of~\eqref{eq:derive-rank-constraint} are not complete.

The \textbf{main results}
of our paper characterize the four ideals appearing in~\Cref{fig:ideals} for all $m$ and $n.$
This square, although not comprising the full atlas, is a natural point of departure since it includes both $\Aqp^{m,n}$ and the multiview ideal $\Abarp^{m,1}.$

In the case of the triangulation ideals $\Abarqp^{m,n}$ and $\Abarp^{m,n},$ ~\Cref{thm:sum-ideal} and \Cref{thm:sum-ideal-GB}, identify simple, explicit generators and Gr\"{o}bner bases under suitable genericity assumptions. For the bundle adjustment ideals  $\Aqp^{m,n}$ and $\Ap^{m,n},$ our characterization is more subtle.
\Cref{thm:saturation} determines these ideals up to saturation by certain minors of the matrix of indeterminates $\stacked{A_1}{A_m}.$

Studying the bundle adjustment ideals with their corresponding triangulation ideals is revealing in interesting ways. 
Despite being closely related, many facts about the triangulation ideals do not transfer to their bundle adjustment ideals---see~Examples~\ref{ex:Ap-4}, \ref{ex:Ap-n-2-m-2}, and~\ref{ex:Ap-n-3-m-2}. The seemingly simple act of making the cameras symbolic gives the bundle adjustment ideals a more elaborate structure. It also offers further insight into the structure of the triangulation ideals. For instance, it explains the origin of the quadrifocal tensor.

The quadrifocal tensor and the associated 4-view constraints~\cite{shashua2000structure} are an oddity in multiview geometry, where they are known to be redundant -- i.e. they are not needed to cut out the multiview variety. However, they are needed to form a Gr\"obner basis for the multiview ideal~\cite{AST13,APT21}.
We show in~\Cref{ex:Ap-4} that, unlike the case of $\Gamma^{m,1}_{\bar{\bA},\pp}$, four-view constraints are necessary to cut out the variety $\GammaAp^{m,1}$, and that these constraints only become redundant after specialization to a particular arrangement of cameras.

The structure of the ideals $\Ap^{m,1}$ and $\Aqp^{m,1}$ plays an important role in our study.
In~\Cref{sec:Ap}, we determine explicit generators for $\Ap^{m,1}$, which form a Gr\"{o}bner basis for a large class of term orders. 
Using this result, we then recover the strongest known results about the multiview ideal $I^{m,1}_{\bar{\bA},\pp}$ as simple corollaries using specialization arguments.
\Cref{sec:Aqp} and~\Cref{sec:Abarqp} follow the same pattern, first giving an explicit Gr\"{o}bner basis for $\Aqp^{m,1}$, and then specializing to obtain $\Abarqp^{m,1}$. 
We then establish our main results in~\Cref{sec:n-points} as consequences of the $n=1$ cases treated in Sections~\ref{sec:Ap}--\ref{sec:Abarqp}. 
Closing the paper, we list several open problems and further research directions centered around the atlas from \Cref{fig:full-diagram}.

\subsection{Notation}
The notation $p \sim p'$ is used to denote equality of $p$ and $p'$ in projective space.

In sections ~\ref{sec:Ap}--\ref{sec:Abarqp}, we restrict attention to the case of a single world point 
$q \in \PP^3$ whose image in the $i$th camera $A_i$, for $i=1,\ldots,m$, is $p_i \in \PP^2$. 
We will use MATLAB notation to refer to subentities of $A, q$ and $p$: 
$A_i[j,k]$ for the entry of matrix $A_i$ in row $j$ and column $k$, $q[j]$ for the $j$th entry 
of $q$ and 
$p_i[j]$ for the $j$th entry of $p_i$.
When 
multiple world points are involved we let $p_{ij}$ be the image of $q_j$ in $A_i$ for each $j=1,\ldots , n.$

The reader has already encountered our use of a bar over an object to indicate specialization. 
For instance, $A$ stands for a symbolic $3 \times 4$ matrix denoting a camera while $\bar{A}$ is a $3 \times 4$ scalar matrix realizing a camera. 
Boldface font will be reserved for sets of variables or ordered row indices of submatrices.
For example, if $\emptyset \neq \row \subseteq \{1,2,3\}$ is a collection of row indices 
of $A_i$, we let $A_i[\row,:]$ denote the $|\row| \times 4$ submatrix of $A$ 
consisting of the rows indexed by $\row$ in the usual order. 
The notation $\bA, \mathbf{q}, \mathbf{p}$ will denote collections
of $m$ cameras, $n$ world points and the corresponding $mn$ image points and $\bar{\bA}, \bar{\mathbf{q}}, \bar{\mathbf{p}}$ 
will refer to their specializations.

\subsection*{Acknowledgements}
We thank Jessie Loucks Tavitas, Erin Connelly \& Craig Citro for helpful discussions. Timothy Duff acknowledges support from the National Science Foundation Mathematical Sciences Postdoctoral Research Fellowship (DMS-2103310). Max Lieblich was partially supported by a National Science Foundation Grant (DMS - 1902251). Rekha Thomas was partially supported by a National Science Foundation grant (DMS - 1719538).

\section{Tools}\label{sec:notation-background}
\subsection{Vanishing Ideals}
The homgoeneous coordinate ring of the product of projective spaces $\PP^{n_1} \times \cdots \times \PP^{n_k}$ is the polynomial ring $\CC [\xx_1, \ldots , \xx_k]$ in the $k$ sets of indeterminates $\xx_i = \{ x_{i1}, \ldots , x_{i(n_i+1)} \}$, which we equip 
with the multigrading $\deg (x_{ij}) = e_i \in \ZZ^k.$ 
A set of polynomials $\{ f_1, \ldots , f_s\} \subset \CC [\xx_1, \ldots , \xx_k]$, 
where each $f_i$ is homogeneous with respect to the multigrading, has a 
well-defined vanishing locus:
\[
\mathrm{V} (f_1, \ldots, f_s) = \{ x \in \PP^{n_1} \times \cdots \times \PP^{n_k} \suchthat f_1 (x) = \cdots = f_s (x) = 0  \}.
\]
These algebraic varieties are the closed sets in the Zariski topology on $\PP^{n_1} \times \cdots \times \PP^{n_k}$.
The polynomials $f_1,\ldots , f_s$ are said to vanish on a set $X$ if $X \subset \mathrm{V} (f_1, \ldots , f_s).$

\begin{definition}
\label{def:vanishing ideal}
The vanishing ideal $\ideal(X)$ of a set $X\subset \PP^{n_1} \times \cdots \times \PP^{n_k}$ 
is the ideal  generated by all $f \in \CC[\xx_1, \ldots , \xx_k]$ vanishing on $X.$
\end{definition}

For polynomials $f_1, \ldots , f_s \in \CC [\xx_1, \ldots , \xx_k]$, we use the standard notation $\langle f_1, \ldots , f_s \rangle$ for the ideal they generate.
The ideal $\langle f_1, \ldots , f_s \rangle$ need not be a vanishing ideal in the sense of~\Cref{def:vanishing ideal}, although it will always be contained in the vanishing ideal 
$\ideal(\V (f_1, \ldots , f_s))$.
A necessary condition for the equality of these ideals is that $\langle f_1, \ldots , f_s \rangle$ is radical. 
We will rely on the following Gr\"obner basis result to 
certify that an ideal is radical. For the basics of Gr\"obner basis theory, 
see \cite{CLO15}.

\begin{proposition}\label{prop:radical}
Let $<$ be a monomial order on the polynomial ring $R[\mathbf{x}] = R[x_1, \ldots , x_k]$ where 
$R$ is an integral domain. 
Suppose $g_1, \ldots , g_s \in R[\mathbf{x}]$ form a Gr\"{o}bner basis with respect to 
$<$, and their leading terms $in_< (g_1), \ldots , in_< (g_s)$ are squarefree monomials in the variables $x_1,\ldots , x_k$.
Then
\begin{enumerate}
    \item $\langle g_1, \ldots , g_s \rangle $ is radical, and
    \item when $R= \mathbb{R}$ is the field of real numbers, $\langle g_1, \ldots , g_s \rangle $ is real radical.
\end{enumerate}
\end{proposition}

For a proof of the first statement, see~\Cref{sec:proofs appendix}, and for the second statement, see 
~\cite[Proposition 1.2]{Vin12}.

Gr\"obner bases satisfying the conditions of \Cref{prop:radical} 
establish much more than the radicality of the ideal they generate.
For particular monomial orders, they serve the extra purpose of deducing results about specializations and 
eliminations of these ideals.
They also imply a nice bonus for ideals defined over $\RR$, namely, 
~\Cref{prop:radical} (2) establishes that the ideal generated by the Gr\"obner basis is 
\emph{real radical}. 
In other words, this ideal is also the vanishing ideal of all real points in 
its variety.

In our applications of \Cref{prop:radical}, $R$ is almost always a field, where many standard results about Gr\"{o}bner bases~\cite[Ch 2]{CLO15} are at our disposal.
The sole exception is in the proof of~\Cref{thm:saturation}, where $R$ is obtained by localizing a polynomial ring at the powers of certain polynomials 
in $\CC [\bA].$

As a reminder, a main contribution of this paper is to explicitly describe the following four vanishing ideals and their corresponding varieties (see~\Cref{def:mother-ideal},~\eqref{eq:GammaAqp}--~\eqref{eq:GammaAbarp}): 
\[
\Aqp^{m,n} = \ideal\left(\GammaAqp^{m,n}\right), 
\quad 
\Ap^{m,n} = \ideal\left(\GammaAp^{m,n}\right) , 
\quad 
\Abarqp^{m,n} = \ideal\left(\GammaAbarqp^{m,n}\right), 
\quad 
\Abarp^{m,n} = \ideal\left(\GammaAbarp^{m,n}\right).
\] 
These ideals live in the polynomial rings $\CC[\bA, \qq, \pp], \CC[\bA,\pp], 
\CC[\qq,\pp]$ and $\CC[\pp]$ respectively.

Throughout the paper, we employ the following simple Nullstellensatz-based \emph{recognition criterion} 
to determine vanishing ideals.
\begin{proposition}[Recognition Criterion]
\label{prop:prop-recognition}
Given a closed subvariety $X\subset \PP^{n_1} \times \cdots \times \PP^{n_k}$, 
its vanishing ideal $\ideal(X)$ is generated by the $\ZZ^k$-homogeneous polynomials $f_1,\ldots , f_s \in \CC[\xx_1, \ldots , \xx_k]$ if and only if the following three conditions hold:
\begin{enumerate}
\item $f_1, \ldots , f_s$ cut out $X$ set-theoretically, i.e., $\V (f_1, \ldots ,f_s) = X.$
\item $\langle f_1, \ldots , f_s \rangle $ is saturated with respect to the irrelevant ideal:
\[
\langle f_1 , \ldots , f_s \rangle = \langle f_1, \ldots , f_s \rangle : (\frakm_{\xx_1} \cap \cdots \cap \frakm_{\xx_k})^\infty 
\]
where $\frakm_{\xx_i} = \langle x_{i1} , \ldots , x_{i(n_i+1)} \rangle .$
\item $\langle f_1, \ldots , f_s \rangle $ is radical.
\end{enumerate}
\end{proposition}
This recognition criterion is a standard consequence of Hilbert's Nullstellensatz, but we 
provide a proof in \Cref{sec:proofs appendix} for completeness. In all of our results, we establish Condition 3 in~\Cref{prop:prop-recognition} 
via explicit Gr\"{o}bner bases as in \Cref{prop:radical}. 

If $M$ is a $m\times n$ matrix over a polynomial ring, then for $\ell \le \min (m, n)$ we let $\minors (\ell,M)$ denote the ideal generated by all $\ell \times \ell$ minors of $M$. We will consider matrices $M$ in which all entries 
are polynomials, some of which may be variables or scalars. 
Here is a small example illustrating some of the above tools. 

\begin{example}{\cite[Example 1.2]{Cox08}}\label{ex:sec-2}
Let $f_1, f_2, f_3$ be the $2\times 2$ minors of the matrix \[
M = \left( \BM a_1 & a_2 & a_3\\ b_1^2 & b_1 b_2 & b_2^2  \EM \right)
\]
in the polynomial ring $\CC [a_1, a_2, a_3, b_1, b_2].$
Explicitly,
$$
f_1 =b_1 (a_1 b_2 - a_2 b_1), \quad 
f_2 = a_1 b_2^2 - a_3 b_1^2, \quad 
f_3 = b_2 ( a_2 b_2 - a_3 b_1).    
$$
The ideal $J = \langle f_1,f_2, f_3 \rangle = \minors(2,M)$ is homogeneous with respect to the $\ZZ^2$-grading $\deg (a_i) = e_1,$ $\deg (b_i)= e_2.$ However it is not the vanishing ideal of 
$\V(f_1,f_2,f_3)$. Indeed, it does not satisfy Condition 2 in~\Cref{prop:prop-recognition}; for instance the polynomial $a_1 b_2 - a_2 b_1$ is an element of $J: (\langle a_1, a_2, a_3 \rangle \cap \langle b_1, b_2 \rangle)^\infty $, but not $J$.
Condition 3 also fails; $f = b_2 (a_2^2 - a_1 a_3) \notin J$,
while $f^3 \in J$.

The vanishing ideal of $\V (J) \subset \PP^2 \times \PP^1$ turns out to be 
\[I = \langle a_3 b_1 - a_2 b_2, a_1 b_2 - a_2 b_1, a_1 a_3 - a_2^2\rangle .\]
The three generators of $I$ form a Gr\"{o}bner basis for both  Lex and GRevLex orders 
in the variable ordering $b_2 < b_1 < a_3 < a_2 < a_1.$
The GRevLex initial ideal $\langle a_3 b_1 , a_2 b_1, a_2^2 \rangle $ is not squarefree.
However, the Lex initial ideal $\langle a_2 b_2, a_1 b_2, a_1 a_3\rangle $ is squarefree, so~\Cref{prop:radical} lets us verify that $I$ is radical.
\end{example}

To prove that Condition 2 in~\Cref{prop:prop-recognition} holds, we make frequent use of the following lemma about GRevLex Gr\"{o}bner bases.

\begin{lemma}[{\cite[Lemma 12.1]{Stu96}}]
\label{lem:grevlex}
Let $G$ be a Gr\"{o}bner basis for an ideal $I\subset \CC [x_1, \ldots , x_l]$
with respect to the GRevLex order where $x_l$ is the cheapest variable:
\[
x_l < \cdots < x_1.
\]
Then
\[
G' = \{ g \in G\suchthat x_l \text{ does not divide } g \} \cup \{ (g/x_l) \suchthat g \in G \textup{ and }  x_l \text{ divides } g\} 
\]
is a GRevLex Gr\"{o}bner basis for the ideal quotient $I: x_l .$ 
\end{lemma}

\subsection{Generic Cameras}\label{sec:genericity}
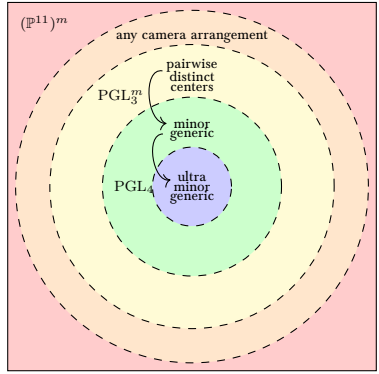
\begin{figure}[t]
\begin{center}
\def\figscale{0.7}
\begin{tikzpicture}[scale=\figscale, every node/.style={scale=\figscale}]
\coordinate (O) at (0,0);
\draw[fill=red!20] (-3.5,-3.5) rectangle ++ (7,7);
\draw[fill=orange!20,dashed] (O) circle (3.35);
\draw[fill=yellow!20,dashed] (O) circle (2.7);
\draw[fill=green!20,dashed] (O) circle (1.7);
\draw[fill=blue!20,dashed] (O) circle (.75);
\node (n1) at (0,.2) {\scriptsize ultra};
\node (n12) at (0, 0) {\scriptsize minor};
\node (n13) at (0, -.2) {\scriptsize generic};
\node (n2) at (0,1.2) {\scriptsize minor};
\node (n21) at (0,1.0) {\scriptsize generic};
\node (n3) at (0,2.3) {\scriptsize pairwise};
\node (n31) at (0,2.1) {\scriptsize distinct};
\node (n32) at (0,1.9) {\scriptsize centers};
\node (n4) at (0,2.85) {\scriptsize any camera arrangement};
\node (n5) at (-2.8,3.1) {\scriptsize $(\PP^{11})^m$};

\node (h2) at (-1,-1) {};
\node (h3) at (-1,-1) {};
\draw[->] (n21) to [in=190,out=180] (n1) node [align=left,midway, label=left:{\scriptsize $\PGL_4 \, \, \, \, \quad $}] {};
\draw [->] (n31) to [in=180,out=170] (n2);
\node (e) at (-1.3,1.7) {\scriptsize $\PGL_3^m \, \, $};
\end{tikzpicture}
\end{center}
\caption{Classes of camera arrangements related by inclusion and equivalence up to projective change-of-coordinates.}\label{fig:camera-space}
\end{figure}

Broadly speaking, the main results of this paper may be divided into three progressively stronger categories: \emph{geometric}  results  giving set-theoretic equations for the varieties of interest, \emph{ideal-theoretic} results which characterize the vanishing ideals of these varieties, and, strongest of all, results about the \emph{Gr\"{o}bner bases} of these vanishing ideals.
For the geometric and ideal-theoretic results, our genericity hypotheses will be that the camera arrangement $\bar{\bA}$ has pairwise distinct centers. 
For the Gr\"obner basis results we will need stronger  
notions of genericity which we discuss in this section.

\begin{definition}[Minor Generic]\label{def:minor-generic}
A camera arrangement $\bar{\bA}=(\bar{A}_1, \ldots , \bar{A}_m)$ is  \emph{minor generic} if all $4\times 4$ minors of the $4 \times 3m$ matrix $\stacked{\bar{A}_1}{\bar{A}_m}$ are nonzero (cf.~\cite[Sec.~2]{AST13},~\cite[Sec.~3]{APT21}.)
\end{definition}
\begin{definition}[Ultra Minor Generic]
\label{def:ultra-minor-generic}
A camera arrangement $\bar{\bA}=(\bar{A}_1,\ldots \bar{A}_m)$ is \emph{ultra minor generic} if all $k\times k$ minors of the $4 \times 3m$ matrix $\stacked{\bar{A}_1}{\bar{A}_m}$ are nonzero for any $k \in [4].$
\end{definition}

We now have three notions of genericity for camera arrangements: pairwise distinct centers, minor generic, and ultra minor generic. They are progressively stronger in 
the sense established by the following proposition, whose proof may be found in~\Cref{sec:proofs appendix}.

\begin{proposition}
If a camera arrangement $\bar{\bA} = (\bar{A}_1,\ldots \bar{A}_m)$ is ultra minor generic 
then it is minor generic, and if $\bar{\bA}$ is minor generic, then it has pairwise distinct centers.
\end{proposition}

We next show that a converse statement of sorts is possible under suitable group actions. This is made precise in \Cref{thm:orbits} and 
summarized pictorially in ~\Cref{fig:camera-space}. 

Consider the following group actions on a camera arrangement $\bar{\bA} = (\bar{A}_1, \ldots \bar{A}_m)$: 
\begin{align}\label{eq:small group-action}
\begin{split}
(\PGL_3)^m &\times (\PP^{11})^m  \to (\PP^{11})^m\\
(H_1, \ldots , H_m) &\, \bar{\bA} = (H_1 \bar{A}_1, \cdots , H_m \bar{A}_m), 
\end{split}
\end{align}
\begin{align}\label{eq:PGL4 group-action}
\begin{split}
\PGL_4 &\times (\PP^{11})^m  \to (\PP^{11})^m\\
H &\, \bar{\bA} = (\bar{A}_1 H^{-1}, \cdots , \bar{A}_m H^{-1}), 
\end{split}
\end{align}
and their combined action:
\begin{align}\label{eq:big group-action}
\begin{split}
\PGL_4 \times (\PGL_3)^m &\times (\PP^{11})^m  \to (\PP^{11})^m\\
\left(H, (H_1, \ldots , H_m)\right) &\, \bar{\bA} = (H_1 \bar{A}_1 H^{-1}, \cdots , H_m \bar{A}_m H^{-1}). 
\end{split}
\end{align}

The combined group action may be interpreted as a simultaneous change of coordinates in both the world $\PP^3$ and in the set of images $(\PP^2)^m$.
We will say that two camera arrangements are equivalent if they lie in the same orbit under one of these group actions.

\begin{theorem} \label{thm:orbits}
\begin{enumerate}
    \item A camera arrangement $\bar{\bA}$ has pairwise distinct centers if and only if it is equivalent to a minor generic camera arrangement under the group action  in~\eqref{eq:small group-action}.
    \item A camera arrangement $\bar{\bA}$ is minor generic if and only if it is equivalent to an ultra minor generic arrangement under the group action in~\eqref{eq:PGL4  group-action}.
    \item A camera arrangement $\bar{\bA}$ has pairwise distinct centers if and only if it is equivalent to an ultra minor generic camera arrangement under the group action in~\eqref{eq:big group-action}. 
\end{enumerate}
\end{theorem}

Unlike the condition of pairwise distinct centers, minor genericity and 
ultra minor genericity are algebraic conditions which are not preserved under coordinate change.  These algebraic genericity assumptions play a 
central role in deducing results about specializations of the bundle adjustment
ideals that we study in this paper.

\section{$\Ap^{m,1}$ : Constraints on cameras and image points for $n=1$}\label{sec:Ap}

In this section we study the ideal $\Ap^{m,1}$ 
consisting of all polynomials that constrain 
a tuple of $m$ symbolic cameras
and the images of a single world point in them.  
Given one camera $A$ and a point $p \in \PP^2,$ there is always a world point $q \in \PP^3$ such that $A q \sim p.$
Therefore, the ideal $\Ap^{1,1}$ is trivial, and we will assume throughout this section that we have $m\ge 2$ cameras.

Let $\bA = (A_1, \ldots, A_m)$ and $\pp = (p_1, \ldots, p_m)$ represent $m$ symbolic cameras and image points in them.

\begin{definition}\label{def:k-focal}
A $k$-{\bf focal} of the pair $(\bA,\pp)$ is any 
polynomial in $\CC [\bA, \pp ]$ that arises as a 
maximal minor of the $3k \times (4+k)$ matrix:
\begin{align}
  \label{eq:k-focal-matrix}
\left(\BM A_{\si_1} &p_{\si_1} &0&\dots &0\\A_{\si_2} &0&p_{\si_2}&\ddots &0 \\ \vdots & \vdots &\ddots&\ddots &\vdots \\A_{\si_k} &0&\dots&0&p_{\si_k}\EM \right), 
\end{align}
where $\sigma = \{ \sigma_1 , \ldots , \sigma_k \}$ is any $k$-element subset of $[m] = \{ 1, \ldots , m \}$ with $k\ge 2.$ 
\end{definition}
Since $4+k\le 3k$ for all $k \ge 2,$ a $k$-focal is specified by a choice of $4+k$ rows of the matrix~\eqref{eq:k-focal-matrix}.
Note that all $k$-focals of $(\bA,\pp)$ lie in the 
ideal $\Ap^{m,1}$ since if there exists a world point $q$ such that $A_i q \sim p_i$ 
for all $i \in [m]$, then the matrix~\eqref{eq:k-focal-matrix} is rank deficient and hence 
all maximal minors of~\eqref{eq:k-focal-matrix} are zero. 

We now state the main result of this section, which establishes the plausible, but by no means obvious, fact that the $k$-focals generate $\Ap^{m,1}.$
Moreover, a distinguished subset of these $k$-focals form a Gr\"{o}bner basis with respect to a large class of monomial orders.

\begin{theorem}
  \label{thm:Ap-GB}
 For any number of cameras $m$, and $n=1$ world point, the $\kfour$-focals form a Gr\"{o}bner basis for $\Ap^{m,1}$ with respect to any product order $<$ with $\bA < \pp .$
\end{theorem}

We will recall the notion of a product order before proving~\Cref{thm:Ap-GB}. 
First, we explain why the $\kfour$-focals are the only ``interesting" $k$-focals.

For the sake of more compact notation, we let
\begin{equation}\label{eq:k-focal-alternate}
\kfoc  = \left(\BM A_{\si_1}[\row_1, :] &p_{\si_1}[\row_1] &0&\dots &0\\A_{\si_2}[\row_2, :] &0&p_{\si_2}[\row_2]&\ddots &0 \\ \vdots & \vdots &\ddots&\ddots &\vdots \\A_{\si_k}[\row_k, :] &0&\dots&0&p_{\si_k}[\row_k]\EM \right), 
\end{equation} 
denote the $k$-focal submatrix of~\eqref{eq:k-focal-matrix} determined by $\sigma$ and the sets $\row_i \subseteq \{1,2,3\}$ indexing the chosen rows of $A_{\sigma_i}$ and $p_{\sigma_i}.$ 
Thus, every $k$-focal is the determinant of a matrix of the form in~\eqref{eq:k-focal-alternate}. 

The following ``bumping up/bumping down'' lemma establishes that all $k$-focals are elementary consequences of $2$-focals, $3$-focals, and $4$-focals.
For $\Abarp^{m,1},$ this observation is a classical result, dating back at least to~\cite[Sec.~5]{Tri95} (see also, eg.,~\cite[Proposition 2]{DBLP:conf/iccv/TragerHP15}, or~\cite[Lemma 2.2]{APT21}.)

For $2\le k-1< m,$ a $(k-1)$-focal $\det \kfoc $
can be ``bumped-up'' to a $k$-focal by including one row (say $j$th) from a new camera $A_i$ with $i \notin \sigma $ as follows:
\begin{equation}\label{eq:bumped-kfocal}
\det
\left(\BM
\kfoc & 0\\
A_i[j,:] & p_i[j]
\EM\right)
.
\end{equation}

There are $3(m-k)$ choices for the pair $(i,j).$
\begin{proposition}
\label{prop:bump}
\par (Bumping-up) 
A ``bumped-up'' $k$-focal~\eqref{eq:bumped-kfocal} can be factored as 
\begin{equation}\label{eq:bump}
\det
\left(\BM
\kfoc & 0\\
A_i[j,:] & p_i[j]
\EM\right)
=
p_i[j] \cdot \det \kfoc ,
  \end{equation}

\par and (Bumping-down) 
every nonzero $k$-focal for $k> 4$ has the form~\eqref{eq:bump} (up to sign.)
\end{proposition}

\begin{proof}
\eqref{eq:bump} is easily verified by Laplace expansion of the determinant.
To prove bumping-down, first note that every term in any $k$-focal has degree $4$ in the camera variables since the camera variables occur in exactly four columns of the matrix \eqref{eq:k-focal-matrix}. Furthermore, this $k$-focal is homogeneous of degree $(\# \row_i -1)$ in the variables of the camera $A_i$. To see this, scale the $k$ cameras by 
$c_1,\ldots ,c_k\in \CC$ and use multi-linearity of $\det $ in rows and columns to obtain
\begin{align*}
\det \big( (c_1,\ldots , c_k) \cdot \bA \, \big| \, \pp \big) [\row ]_\sigma &= 
\left(\displaystyle\prod_{i=1}^{k} c_i^{\# \row_i} \right) \cdot \det \big( \bA \, \big| \, (c_1^{-1}, \cdots , c_k^{-1}) \cdot \pp \big) [\row ]_\sigma\\
&= \left(\displaystyle\prod_{i=1}^{k} c_i^{\# \row_i-1} \right) \cdot \det \kfoc .
\end{align*}

Since $4=\sum_{i=1}^{k} (\# \row_i -1),$ and $\# \row_i -1 \geq 0$ for all $i$, 
at most four cameras can use more than one row.
Thus, if $k > 4,$ there exists a camera $A_i$ using a single row indexed by 
$\row_i = \{ j \},$ which implies the factorization in~\eqref{eq:bump}.
\end{proof}

As explained in~\cite[Ch.~17]{HZ04}, the $\kfour$-focals encode many familiar entities from multiview geometry. A $2$-focal has the form
\[
\left( \bA \, |\, \pp \right) [\{1 , 2, 3 \}, \{ 1, 2, 3 \} ]_{\{ i, j \}} = \det
\left(\BM {A}_{i} &p_{i} &0\\
{A}_{j} &0&p_{j}\EM\right)
  = p_j^\top \, F_{i j} ({A}) \, p_i,
\]
where $F_{ij} (A)$ is the $3\times 3$ \emph{fundamental matrix} associated to the pair $({A}_i, \, {A}_j).$
Similarly, for $\{ j_1, j_2 \}, \{ k_1, k_2 \} \subset \{1,2,3\},$ the $3$-focal
\begin{align*}
&\det \left( \bA \, | \, \pp \right) [\{ 1, 2, 3 \}, \{j_1, j_2\}, \{ k_1, k_2 \} ]_{\{ i, j, k \}} \\
&=\det \left(\begin{matrix*}[r] {A}_{i} &p_{i} &0 & 0 \\
{A}_j[ \{j_1,j_2\}, :]
&0&p_j[ \{ j_1, j_2\},:]&0\\ {A}_k[ \{k_1,k_2\}, :] &0&0&p_k [ \{k_1, k_2\},:]\\ \end{matrix*}\right)
\end{align*}
may be viewed as a polynomial in $\CC [\pp]$ whose coefficients in $\CC[\bA]$ make up a particular \emph{trifocal tensor} for the triple $({A}_i, \, {A}_j, \, {A}_k).$
Note in the equation above the privileged index $i$, also known as the covariant index; for each triple $(A_i, A_j, A_k),$ there are $27$ choices of $3$-focals where the covariant index may be $i, j,$ or $k.$
Finally, each $4$-tuple $({A}_i, \, {A}_j, \, {A}_k, \, {A}_l)$ gives rise to $81$ $4$-focals
\begin{align*}
&\det \left( \bA \, | \, \pp \right) [\{ i_1, i_2 \}, \{ j_1, j_2 \}, \{ k_1, k_2 \} \{ l_1, l_2 \} ]_{\{ i, j, k, l \}}\\
&= \det \left(\begin{matrix*}[r] {A}_{i} [ \{i_1,i_2\}, :] &p_{i} [ \{ i_1, i_2 \} , : ] &0 & 0 & 0\\
{A}_j[ \{j_1,j_2\}, :]
&0&p_j[ \{ j_1, j_2\}, :]&0 & 0\\ {A}_k[ \{k_1,k_2\}, :] &0&0&p_k [ \{k_1, k_2\}, :] &0\\ 
{A}_l[ \{l_1,l_2\}, :] &0&0&0 &p_l [ \{l_1, l_2\}, :]
\end{matrix*}\right),
\end{align*}
whose coefficients in $\CC[\bA]$ form the $81$ \emph{quadrifocal tensors.}

We will use \Cref{prop:prop-recognition} to 
prove~\Cref{thm:Ap-GB}. This in turn will rely on \Cref{prop:radical}. 
We begin by proving that both the $m$-focals and the $\kfour$-focals cut out $\GammaAp^{m,1}$ \emph{set-theoretically.}
The latter will establish Condition 1 of the recognition criterion.

\begin{proposition}
\label{prop:Ap-set}
\[ \GammaAp^{m,1}  =  \V (\kfour\text{-focals}) =  \V (m\text{-focals}).
\]
\end{proposition}

\begin{proof}
We first show that  $\GammaAp^{m,1}  \subseteq  \V (\kfour\text{-focals})$. 
Suppose $(\bar{\bA}, \bar{\pp}) \in \GammaAp^{m,1},$ and fix a representation for this point in homogeneous coordinates. 
Then there exists a nonzero $4\times 1$ matrix $\bar{q}$ and scalars $\lambda_{1},\ldots , \lambda_m$ with $\bar{A_i} \bar{q} = \lambda_{i} \bar{p_{i}}.$
This implies that the kernel of each $k$-focal matrix~(\ref{eq:k-focal-matrix}) associated to $(\bar{\bA}, \bar{\pp})$ contains the nonzero vector
\begin{equation}
  \label{eq:null}
\left(\BM - \bar{q}  &  \lambda_{\sigma_1} & \cdots & \lambda_{\sigma_k } \EM\right)^\top.
\end{equation}
Thus all $k$-focals, in particular all $\kfour$-focals, vanish at $(\bar{\bA}, \bar{\pp})$.

The containment $\V (\kfour\text{-focals}) \subseteq  \V (m\text{-focals})$ 
follows from~\Cref{prop:bump}.

To finish, we argue that $\V (m\text{-focals}) \subseteq \GammaAp^{m,1}$. Consider a point $$(\bar{\bA}, \bar{\pp}) = (\bar{A}_1, \ldots , \bar{A}_m, \bar{p}_1, \ldots , \bar{p}_m) \in \V (m\text{-focals}),$$
with fixed homogeneous coordinates as before, and consider the vector~\eqref{eq:null} for $\sigma = [m]$.  To prove this inclusion we will 
construct a sequence of points 
$$(\bar{A}_1^{(n)}, \ldots , \bar{A}_m^{(n)}, \bar{p}_1, \ldots , \bar{p}_m) \in \GammaAp^{m,1}$$
converging to $(\bar{\bA}, \bar{\pp}).$ 
Since $\GammaAp^{m,1}$ is closed in the Euclidean topology, this will imply that $(\bar{\bA}, \bar{\pp}) \in \GammaAp^{m,1}$, giving the desired inclusion.
To construct the sequence, set $\bar{A}_i^{(n)} = \bar{A}_i$ for any $i$ with $\lambda_i \ne 0.$
For each index $i$ with $\lambda_i =0,$ choose an arbitrary camera $\bar{A}_i '$ such that $\bar{A}_i ' \bar{q} \sim \bar{p}_i$ 
and set $\bar{A}_i^{(n)} = \bar{A}_i + (1/n) \bar{A}_i '.$
This gives a sequence of points \[
(\bar{A}_1^{(n)} , \ldots , \bar{A}_m^{(n)}, \bar{A}_1^{(n)} \bar{q}, \ldots , \bar{A}_n^{(n)} \bar{q})_{n \ge 1}\]
in $\GammaAp^{m,1}.$
To see that it converges to the given point $(\bar{\bA}, \bar{\pp}) \in \GammaAp^{m,1} ,$ we may consider each camera separately: if $\lambda_i \neq 0$, then $\bar{A}_i^{(n)} = \bar{A}_i$ and
\[
\bar{A}_i^{(n)} \bar{q} = 
\bar{A}_i \bar{q} \sim  \bar{p}_i, 
\]
for all $n,$
whereas if $\lambda_i = 0$, then $\bar{A}_i^{(n)} \to \bar{A}_i$ as $n\to \infty $, and
\[
\bar{A}_i^{(n)} \bar{q} =  \bar{A}_i \bar{q} + \frac{1}{n} \bar{A}_i' \bar{q} = \frac{1}{n} \bar{A}_i' \bar{q} \sim \bar{A}_i' \bar{q} \sim \bar{p}_i.
\]
\end{proof}

\Cref{prop:Ap-set} shows that the ideals of  $m$-focals and $\kfour$-focals both cut out $\GammaAp^{m,1}$ set-theoretically; therefore, by~\Cref{prop:prop-recognition},
\begin{align*}
  \Ap^{m,1} &= \sqrt{\langle \kfour\text{-focals} \rangle : (\frakm_\bA \cap \frakm_\pp)^\infty }  \\
  &= \sqrt{\langle m\text{-focals} \rangle : (\frakm_\bA \cap \frakm_\pp)^\infty }.
\end{align*}
\Cref{thm:Ap-GB} will establish the following precise relationship between the three ideals:
\begin{equation}\label{eq:equality-Ap-focals}
\Ap^{m,1} = \langle \kfour\text{-focals} \rangle  \\
  = \langle m\text{-focals} \rangle : \frakm_\pp^\infty .
\end{equation}

Our next step in the proof of \Cref{thm:Ap-GB} is to note that the $m$-focals form a 
universal Gr\"obner basis for the ideal they generate.
As noted in~\cite{DBLP:conf/iccv/TragerHP15}[Proposition b3], the $m$-focal ideal belongs to 
the class of \emph{sparse determinantal ideals}, studied in algebraic geometry and commutative algebra.
Sparse (maximal) determinantal ideals are generated by the maximal minors of matrices whose nonzero entries are distinct indeterminates.
These ideals were studied by Giusti and Merle~\cite{GM82}, who determined their dimensions and associated primes, and later by Boocher~\cite{Boo12}, who, in addition to determining their minimal free resolutions and several related invariants, showed that the maximal minors form a universal Gr\"{o}bner basis.
The following is an immediate consequence of~\cite{Boo12}[Proposition 5.4] and~\Cref{prop:radical}.

\begin{proposition}
\label{prop:m-focal-UGB}
For $n=1$ world point and any number of cameras $m$, the $m$-focals form a universal Gr\"{o}bner basis for the ideal they generate, and this ideal is radical.
\end{proposition}

\Cref{prop:m-focal-UGB} does not directly give us any information about $\langle \kfour \text{-focals} \rangle ,$ since its generators are minors of several different matrices which share variables. However, \Cref{thm:Ap-GB} asserts that the  $\kfour$-focals form a Gr\"{o}bner basis for a large class of term orders.
To get there, we next develop a result of a similar flavor to~\Cref{lem:grevlex}, about Gr\"{o}bner bases of multihomogeneous ideals in 
$\CC [\xx_1, \ldots , \xx_k, \yy]$ whose elements are linear in each set of variables $\xx_1, \ldots , \xx_k$.
This is~\Cref{lem:colon-wellsupported}, which deals with a product order instead of a GRevLex order.

Recall that two monomial orders $<_\xx  $ on $\CC [\xx]$ and $<_\yy$ on $\CC[\yy]$ specify a \emph{product order} on $\CC[\xx, \yy ]$ with $\yy < \xx $ as follows: $\xx^{\alpha_1} \yy^{\beta_1} < \xx^{\alpha_2} \yy^{\beta_2 }$ when either $\xx^{\alpha_1} <_\xx \xx^{\alpha_2}  $ or $\alpha_1 = \alpha_2$ and $\yy^{\beta_1} <_\yy \yy^{\beta_2} .$
More generally, we may define a product order on $\CC[ \xx_1, \ldots , \xx_k]$ with $\xx_k < \cdots < \xx_1$ by comparing monomials first using a monomial order $<_{\xx_1}$ on $\CC [\xx_1]$ and breaking any ties successively with $<_{\xx_2}, \ldots , <_{\xx_k}.$
We also need the following condition on the monomial supports of a (partially) multilinear polynomial $g.$

\begin{definition}\label{def:well-supported}
Suppose $g \in \CC [\xx_1, \ldots ,\xx_k, \yy]$ is multilinear in $\xx_1, \ldots , \xx_k$ and supported in each group of variables on the subsets $\xx_j' \subset \xx_j$ for $j=1,\ldots , k.$
We say that $g$ is well-supported on $\xx_1, \ldots , \xx_k$ if every monomial of the form $\prod_{j=1}^k x_j$ with $x_j\in \xx_j '$ divides some term of $g.$
\end{definition}

Essentially, $g$ is well-supported if it is multilinear in $\xx_1, \ldots , \xx_k$ and its monomial support in $\xx_1', \ldots , \xx_k '$ (the $\xx$ variables appearing in the support of $g$) is as large as possible.

\begin{example} \label{ex:well-supported}
Consider
\[
g = x_{1 2} x_{2 2} + y^2 x_{1 1} x_{2 1} + 3 x_{1 1} x_{2 2} + 2  y x_{1 2} x_{2 1}
\]
in the polynomial ring $\CC [x_{1 1} , x_{1 2}, x_{1 3}, x_{2 1}, x_{2 2}, y].$
Then $g$ is 
well-supported on $\xx_1 = \{x_{1 1}, x_{1 2}, x_{1 3}\}$ and 
$\xx_2 = \{x_{2 1}, x_{2 2}\}$. The monomial support of 
$g$ comes from $\xx_1 ' = \{ x_{1 1}, x_{1 2} \}$
and $\xx_2 ' = \xx_2 $ and indeed, all four products of variables from
$\xx_1'$ and $\xx_2'$ divide some term of $g.$

Suppose now that we remove the first term from $g$: 
\[
g' = y^2 x_{1 1} x_{2 1} + 3 x_{1 1} x_{2 2} + 2  y x_{1 2} x_{2 1}.
\]
Then $g'$ is not
well-supported on $\xx_1$ and $\xx_2$ since no term of $g'$ is divisible by 
$x_{1 2} x_{2 2}$.
\end{example}

The following fact is used implicitly in the proof of~\cite[Theorem 2.1]{AST13}.
\begin{proposition}\label{prop:well-supported}
Suppose $g \in \CC [\xx_1, \ldots ,\xx_k, \yy]$ is multilinear in 
$\xx_1, \ldots , \xx_k$ and well-supported on $\xx_1, \ldots \xx_k$.
Let $\xx = \xx_1 \cup \cdots \cup \xx_k$, and $<$ be a product order with $\yy < \xx $ formed from $<_\xx $ and $<_\yy .$
Then the leading term $in_< (g) $ depends only on $<_\yy$ and the relative ordering of the variables within each group $\xx_1,\ldots \xx_k$.
\end{proposition}
\begin{proof}
Let $x_1, \ldots , x_k$ be the largest variables with respect to $<_\xx$ in each of the sets $\xx_1', \ldots , \xx_k'$. 
Since $g$ is well-supported on $\xx_1, \ldots , \xx_k$, it has a monomial of the form $(\prod_{j=1}^k x_i) \cdot \yy^{\alpha }.$ 
Choosing $\yy^{\alpha }$ maximal with respect to $<_\yy ,$  $in_< (g) = c\cdot (\prod_{j=1}^k x_i) \cdot \yy^{\alpha }$ for some $c\in \CC.$
\end{proof}

We note that~\Cref{prop:well-supported} may fail when $g$ is not well-supported.
In \Cref{ex:well-supported}, suppose $x_{1 1} < x_{1 2}$ in $\xx_1$ and 
$x_{2 1} < x_{2 2}$ in $\xx_2,$. Then since, $x_{1 2} x_{2 2}$ is not in the monomial support of $g',$ we could have either $in_< (g') = 2 y x_{1 2} x_{2 1}$ or $in_< (g') = 3 x_{1 1} x_{2 2}.$ 

\begin{lemma}\label{lem:colon-wellsupported}
With notation as in~\Cref{prop:well-supported}, let $G$ be a Gr\"{o}bner basis for a $\ZZ^{k+1}$-homogeneous ideal $I$ in $\CC [\xx_1, \ldots , \xx_k, \yy]$ with respect to a product order where $\yy < \xx$ . 
If every element of $G$ is multilinear and well-supported on $\xx_1, \ldots, \xx_k$ 
and $x_i$ is the cheapest variable in $\xx_i$, then 
\[
G' = \{ g \in G\suchthat x_i \text{ does not divide } g \} \cup \{ (g/x_i) \suchthat g \in G \textup{ and }  x_i \text{ divides } g\} 
\]
is a Gr\"{o}bner basis with respect to $<$ for the ideal quotient $I: x_i .$ 
\end{lemma}

\begin{proof}
The proof is analogous to that given for~\Cref{lem:grevlex} in~\cite{Stu96}.
For any $f\in I : x_i$, we have that $x_i f \in I$ and hence $in_<(g)$ divides $x_i in_<(f)$ for some $g\in G.$
We must show that the leading term of some element of $G'$ divides $in_<(f)$. 
We argue by considering two cases.
In the first case, if $x_i$ divides $g,$ then the leading term of $g/x_i\in G'$ divides $in_<(f).$
In the second case, $x_i$ does not divide $g$ and hence $g\in G'.$ 
\Cref{prop:well-supported} then implies that $x_i$ also does not divide $in_<(g)$.
Indeed, if $x_i$ did divide $in_<(g),$ then $x_i$ would be the most expensive variable in $\xx_i '$, and this would force $\xx_i ' = \{ x_i \}.$ By the well-supportedness of $g,$ this would imply that $x_i$ divides $g$, a contradiction.
Thus $x_i$ does not divide $in_<(g)$, which implies $in_< (g)$ divides $in_<(f).$
\end{proof}

Applying~\Cref{lem:colon-wellsupported} over \emph{all} product orders with $\yy < \xx$ yields the following.

\begin{corollary}\label{cor:universal-product}
With notation as in~\Cref{lem:colon-wellsupported}, if $G$ is a Gr\"{o}bner basis with respect to all product orders with $\yy < \xx,$ then the same is true for $G'.$
\end{corollary}

Thus, if $G$ is a Gr\"{o}bner basis with respect to all product orders with $\yy < \xx$, then we obtain a Gr\"{o}bner basis with respect to all product orders for every successive quotient of $\langle G \rangle$ by variables in $\xx.$
With this tool in hand, we are now ready to prove~\Cref{thm:Ap-GB}.

\begin{proof}[Proof of~\Cref{thm:Ap-GB}]
We construct an ascending chain of ideals
  \[
  \langle m\text{-focals} \rangle = J_0 \subset J_1 \subset \cdots \subset J_{m} = \langle \kfour\text{-focals} \rangle
  \]
  where each $J_k = \langle G_k \rangle $ is given by an explicit set of generators $G_k$ which we now define.
  Set $G_0 = \{ m\text{-focals} \}$, and for $k>0,$
\begin{align*}
G_k = & \{ g \in G_{k-1} \suchthat p_k [i] \text{ doesn't divide } g \, \, \forall i \}
 \,\,\displaystyle\cup \\
& \{ (g/p_k[i]) \text{ s.t. } g \in G_{k-1} \textup{ and } 
p_k[i] \text{ divides } g \textup{ for some } i \}.
\end{align*}
Applying~\Cref{prop:bump}, one can also fix a $j \in [3]$ and obtain $G_k$ as 
\begin{align*}
G_k = & \{ g \in G_{k-1} \suchthat p_k[i] \text{ doesn't divide } g \, \, \forall i \} \,\, \displaystyle\cup\\
& \{ (g/p_k[j]) \suchthat  g \in G_{k-1} \textup{ and } p_k[j] \text{ divides } g \}. 
\end{align*}

Let $<$ be any product order with $\bA < \pp .$
We now argue that $G_k$ is a Gr\"{o}bner basis with respect to $<$ for all $k$. This is true for 
$k=0$ by~\Cref{prop:m-focal-UGB}, so suppose $G_{k-1}$ is a Gr\"{o}bner basis 
with respect to $<$ for some $k>0$. 
Then 
\Cref{prop:bump} implies that any element $g\in G_k$ is a $\ell$-focal for some integer $\ell.$
An $\ell$-focal is the determinant of a minor of \eqref{eq:k-focal-matrix} for 
$k = \ell$ involving $4+\ell$ rows. Every term in this determinant is 
multilinear in $\pp_1, \ldots, \pp_m$ and thus involves exactly one variable from 
each subset of supporting variables $\pp_1 ' \subset \pp_1, \ldots , \pp_m ' \subset \pp_m$. Conversely, every possible product of $\ell$ 
variables, one from each $\pp_i '$ appears in a term of the 
$\ell$-focal. Therefore,  $g$ is well-supported on $\pp_1, \ldots , \pp_m$. 
Moreover, each term of the $\ell$-focal is the product of a 
$\pp$-monomial as above and a $4\times 4$ minor of of the symbolic matrix $\stacked{A_1}{A_m}$.

By \Cref{cor:universal-product}, $G_k$ is a Gr\"{o}bner basis with respect to $<$ for the ideal
\begin{equation}
\label{eq:sat-p}
  J_{k} = J_{k-1} : p_k[ 1]  = J_{k-1} : p_k[2] = J_{k-1} :  p_k[3] .
\end{equation}

When $k=m,$~\Cref{prop:bump} establishes the equality $G_m = \{ \kfour \text{-focals} \}$ and hence $J_m = \langle \kfour \text{-focals} \rangle $.
Further, $G_m$ is a Gr\"{o}bner basis with respect to all product orders with $\bA < \pp .$ 

To conclude that $J_m = \Ap^{m,1},$ we verify that $J_m$ satisfies the three conditions of~\Cref{prop:prop-recognition}.
\Cref{prop:Ap-set} gives the set-theoretic Condition 1.
For Condition 3, we must show that $J_m$ is radical; this follows from~\Cref{prop:radical} since the lead terms of the $\kfour$-focals are all squarefree.
Finally, for Condition 2, we calculate

\begin{align*}
J_m &\subset J_m : \frakm_\pp^{\infty } \\
&= J_m : \frakm_\pp \tag{$J_m$ is radical}\\
&\subset  J_m : \left( \displaystyle\prod_{k=1}^m (p_k[1] \cdot p_k [2] \cdot p_k [3]) \right) \\
&= \left( \left( J_m : p_1 [1]\right)  : \cdots \right): p_m [3]\\
&= J_m, \tag{\Cref{cor:universal-product} w/ $G=G_m$}
\end{align*}
forcing $J_m = J_m : \frakm_\pp^{\infty }.$
A similar calculation, shown below, gives $J_m = J_m : \frakm_\bA^{\infty }$:
\begin{align*}
J_m &= J_0 : \frakm_\pp \tag{\Cref{prop:bump}}\\
    &= (J_0 : \frakm_\bA ) : \frakm_\pp \tag{\Cref{prop:bump}, \Cref{prop:m-focal-UGB}}\\
    &= (J_0 : \frakm_\pp ) : \frakm_\bA \\
    &= J_m : \frakm_\bA \\
   &= J_m : \frakm_\bA^{\infty }. \tag{$J_m$ is radical} 
\end{align*}
Thus $J_m = J_m : (\frakm_\bA \cap \frakm_\pp )^\infty  $, and we may conclude from~\Cref{prop:prop-recognition} that
\[
\langle \kfour \text{-focals} \rangle = J_m = \Ap^{m,1}.
\]
\end{proof}

\section{$\Abarp^{m,1}$ : The multiview ideal for $n=1$}\label{sec:Abarp}

We now specialize the camera arrangement $\bA$ and consider the ideal 
$\Abarp^{m,1}$. Recall that the \emph{multiview variety} is the closed image of the following rational imaging map associated to a camera arrangement $\bar{\bA}$: 
\begin{align}\label{eq:imaging map}
\begin{split}
\varphi_{\bar{\bA}} : \PP^3 &\dashrightarrow (\PP^2)^m \\
q &\mapsto (\bar{A}_1 q , \ldots , \bar{A}_m q ).
\end{split}
\end{align}
The closure of the image of $\varphi_{\bar{\bA}}$ is precisely the variety $\GammaAbarp^{m,1}$ in~\eqref{eq:GammaAbarp}. 
Thus, the vanishing ideal $\Abarp^{m,1}$ is none other than the \emph{multiview ideal} of the camera arrangement $\bar{\bA}$~\cite{APT21, AST13}. Much previous work has also gone into studying various polynomials that vanish on $\GammaAbarp^{m,1}$~\cite{DBLP:conf/iccv/FaugerasM95,HA97,ma2004invitation,Tri95,DBLP:conf/iccv/TragerHP15} 

In this section, we show how
Theorem~\ref{thm:Ap-GB} recovers two important previously-known theorems about multiview ideals.
Each of these previous results assumes the camera arrangement $\bar{\bA}$ is generic in one of the senses discussed in~\Cref{sec:genericity}.
The first theorem characterizes a \emph{universal Gr\"{o}bner basis} for the multiview ideal of a minor generic camera arrangement; that is, a Gr\"{o}bner basis with respect to all possible monomial orders on $\CC [\pp ].$

\begin{theorem}[{\cite[Theorem 2.1]{AST13}}]
\label{thm:AST-UGB}
For minor generic $\bar{\bA},$ the specialized $\kfour$-focals from $\Ap^{m,1}$ form a universal Gr\"{o}bner basis for
the multiview ideal $\Abarp^{m,1}.$
\end{theorem}

In the work of Aholt, Sturmfels, and Thomas~\cite{AST13},~\Cref{thm:AST-UGB} is one of the main results needed to study the Hilbert Scheme parametrizing multiview ideals.
We will obtain this result as a direct corollary of~\Cref{thm:Ap-GB}. 
Our argument is essentially independent of the original proof, which requires relating the multiview ideal to yet another Hilbert scheme studied in~\cite{CS10}.
In our proof, Boocher's result in~\Cref{prop:m-focal-UGB} does much of the heavy lifting.
Boocher's result in turn relies on an older result of Bernstein, Sturmfels, and Zelevinsky~\cite{BZ93,SZ93}, the proof of which was later greatly simplified by Conca, De Negri, and Gorla~\cite{CDG15}.
We also note that $\Abarp^{m,1}$, is \emph{not} a sparse determinantal ideal, since the matrices involved have some 
nonzero constant entries. Therefore \Cref{prop:m-focal-UGB} 
cannot be applied directly to prove~\Cref{thm:AST-UGB}.

Recall that minor genericity of $\bar{\bA}$ is stronger than $\bar{\bA}$ 
having pairwise distinct centers. For arbitrary $\bar{\bA}$ with pairwise distinct centers, the Gr\"{o}bner-theoretic~\Cref{thm:AST-UGB} no longer applies.
However, we still have the following ideal-theoretic characterization of $\Abarp^{m,1}.$

\begin{theorem}[{\cite[Theorem 3.7]{APT21}}]\label{cor:2-3-focal}
For $\bar{\bA}$ with pairwise distinct centers, $\Abarp^{m,1}$ is generated by the specialized $\kthree$-focals.
\end{theorem}

The set-theoretic analogue of~\Cref{cor:2-3-focal} is a classical result of multiview geometry.
Indeed, it can be shown under even stronger genericity assumptions on $\bar{\bA}$ that just the $2$-focals cut out $\GammaAbarp^{m,1}$ for $m\ge 4$~\cite{HA97,DBLP:conf/iccv/TragerHP15}, although they do not generate $\Abarp^{m,1}.$

We observe that the $4$-focals, although needed for~\Cref{thm:AST-UGB}, are not needed for~\Cref{cor:2-3-focal}.
Indeed, it is a well-established principle in the vision literature that the $4$-focals ``do
not add more information"~\cite{DBLP:conf/iccv/FaugerasM95} and ``are always completely unnecessary"~\cite{DBLP:conf/iccv/TragerHP15}.
However, these statements apply only to the multiview ideal $\Abarp^{m,1}.$
By contrast, the ideal $\Ap^{m,1}$ is \emph{not} generated by only the $\kthree$-focals.

\begin{example}
 \label{ex:Ap-4}
We investigate $\Ap^{4,1} = \langle \kfour \text{-focals} \rangle$ using the computer algebra system \texttt{Macaulay2}~\cite{M2}.
The containment $\langle \kthree \text{-focals} \rangle \subsetneq \Ap^{4,1}$ is strict, since there is a nonzero remainder upon polynomial division of any $4$-focal $f$ by a Gr\"{o}bner basis for the ideal $\langle \kthree \text{-focals} \rangle$.
However, polynomial division also shows, whenever $i\in [m]$ and $1\le i_1 < i_2 < i_3 \le 4$, that
\begin{equation}\label{eq:sat-23-focal}
\det A_i [:, \{ i_1, i_2, i_3 \} ] \cdot  f \in \langle \kthree \text{-focals} \rangle .
\end{equation}
\eqref{eq:sat-23-focal} implies more generally that for any number of cameras $m$, $\Ap^{m,1}$ may be obtained from successive ideal quotients of $\langle \kthree \text{-focals} \rangle$ by certain subdeterminants of the symbolic camera matrices $A_1, \ldots , A_m.$ 
\end{example}

We reiterate that the $4$-focals are not needed to generate 
$\Abarp^{m,1}$, but they are needed to generate $\Ap^{m,1}$ where the 
cameras are symbolic. As~\Cref{ex:Ap-4} shows, this is because the open condition that each symbolic camera matrix has full rank is not enforced when we consider $\Ap^{m,1}$.
Moreover, for {\em any} camera arrangement $\bar{\bA},$~\Cref{ex:Ap-4} recovers the well-known fact that
\[
\langle \text{specialized } \kthree \text{-focals} \rangle
= 
\langle \text{specialized } \kfour \text{-focals} \rangle .
\]
since, for each $i$, specializing $\bA$ to $\bar{\bA}$ in each of the four polynomials $\det A_i [:, \{ i_1, i_2, i_3 \} ]$ yields four scalars, at least one of which is nonzero.
Bearing this fact in mind, we still work with the $4$-focals in the remainder of this section, as they are necessary for the Gr\"{o}bner-theoretic~\Cref{thm:AST-UGB} and in order to apply the results of the previous section.
Thus, we state the well-known set-theoretic analogue of~\Cref{cor:2-3-focal} as follows.

\begin{proposition}
  \label{prop:multiview-focals}
 If $\bar{\bA}$ is a camera arrangement with pairwise distinct centers, then 
  \[
  \GammaAbarp^{m,1} = \V \left( \langle \text{specialized } \kfour \text{-focals} \rangle \right).
  \]
\end{proposition}

We provide a complete, self-contained proof of \Cref{prop:multiview-focals}
since it is an important step in our proofs of~\Cref{thm:AST-UGB} and~\Cref{cor:2-3-focal}.
Also, we appeal to the geometric argument used in the proof when establishing various results later in the text (eg.~\Cref{thm:Abarqp-gb},~\Cref{thm:ba-set-theoretic}.) 

Before proceeding to the proofs, let us consider three simple camera arrangements and their multiview ideals.
For the first example, camera centers are distinct, so~\Cref{prop:multiview-focals} applies.

\begin{example}
\label{ex:multiview-2-generic}
For $m=2$ cameras with distinct centers, the multiview ideal is a principal ideal which can be obtained by specializing the sole $2$-focal generating $\Ap^{2,1}.$
The $\ZZ^2$-graded Hilbert series of $\Abarp^{2,1}$ in this case equals
\[
\frac{1-T_{\pp_1}T_{\pp_2}}{\left(1-T_{\pp_2}\right)^{3}\left(1-T_{\pp_1}\right)^{3}} = 
1+3\,T_{\pp_1}+3\,T_{\pp_2}+6\,T_{\pp_1}^{2}+8\,T_{\pp_1}T_{\pp_2}+6\,T_{\pp_2}^{2} + \ldots 
\]
Up to coordinate changes in the world and images, such an arrangement takes the form $\bar{\bA} = \left( \left(\BM I & 0 \EM\right), \left(\BM0 &  I  \EM\right)\right).$
Although generic in the sense of having distinct centers, this arrangement is special in the sense that its multiview ideal is \emph{toric}~\cite[Sec 4]{AST13}:
\[
\Abarp^{2,1} = \langle -p_1[2]p_2[2]+p_1[3]p_2[1] \rangle .
\]

\end{example}

The next two examples illustrate why the assumption of distinct centers cannot be relaxed in 
\Cref{prop:multiview-focals}. Namely, when camera centers coincide, the multiview ideal can contain polynomials which are not specializations of polynomials in $\Ap^{m,1}.$
See~\cite{APT21}[Example 3.10] for another instance of this phenomenon.

\begin{example}
\label{ex:multiview-2-coincident}
For $\bar{\bA} = \left( \left(\BM I & 0 \EM\right), \left(\BM I &  0  \EM\right)\right),$ the camera centers coincide.
Its multiview ideal is generated by the $2\times 2$ minors of the $3\times 2$ matrix $\left( \BM p_1 & p_2 \EM \right),$ with Hilbert series 
\[
\frac{1-3\,T_{\pp_1}T_{\pp_2}+T_{\pp_1}^{2}T_{\pp_2}+T_{\pp_1}T_{\pp_2}^{2}}{\left(1-T_{\pp_2}\right
      )^{3}\left(1-T_{\pp_1}\right)^{3}}
      =
      1+3\,T_{\pp_1}+3\,T_{\pp_2}+6\,T_{\pp_1}^{2}+6\,T_{\pp_1}T_{\pp_2}+6\,T_{\pp_2}^{2} + \ldots .
\]
In contrast to the previous example, specializing the $2$-focal generating $\Ap^{2,1}$ to $\bar{\bA}$ results in the zero polynomial. 
Thus, in general, $\Abarp^{m,n}$ cannot be obtained from $\Ap^{m,n}$ by specialization when the 
cameras in $\bar{\bA}$ do not have distinct centers.
\end{example}

\begin{example}
\label{ex:multiview-3-partially-coincident}
For the camera arrangement \[
\bar{\bA} = (\bar{A_1}, \bar{A_2}, \bar{A_3}) = \left( \left(\BM I & 0 \EM\right), \left(\BM I &  0  \EM\right), \left(\BM0 &  I  \EM\right) \right),
\]
the first two camera centers coincide.
The multiview ideal is
\begin{equation}
\begin{split}
\left\langle 
\det \left( \bar{\bA} \mid \pp \right) [1:3, 1:3]_{\{ 1, 3 \}}
,
\,
\det \left( \bar{\bA} \mid \pp \right) [1:3, 1:3]_{\{ 2, 3 \}}
\right\rangle +  \minors \left( 2, \left(\BM p_1 & p_2 \EM\right) \right).
\end{split}
\end{equation}
The first summand above is generated by specializations of the $2$-focals in $\Ap^{3,1}$ corresponding to the camera pairs $(\bar{A_1}, \bar{A_3})$ and $(\bar{A}_2, \bar{A_3})$ with distinct centers.
The second summand is the multiview ideal from the previous example.
Upon specializing $\Ap^{3,1}$ to this particular arrangement $\bar{\bA},$ the $2$-focal on $(A_1,A_2)$ and all $3$-focals in $\Ap^{3,1}$ become zero, in contrast to the case of a generic camera triple.
\end{example}

\Cref{cor:2-3-focal}~asserts that the failure to obtain $\Abarp^{m,1}$ by specialization from $\Ap^{m,1}$ occurs \emph{only} in examples with coincident centers like in~\Cref{ex:multiview-2-coincident} and~\Cref{ex:multiview-3-partially-coincident}.
We begin by proving the set-theoretic version.

\begin{proof}[Proof of~\Cref{prop:multiview-focals}]
The inclusion $\GammaAbarp^{m,1} \subset \V \left( \langle \text{specialized } \kfour \text{-focals} \rangle \right)$ 
follows along similar lines as the inclusion  $\GammaAp^{m,1}  \subseteq  \V (\kfour\text{-focals})$ in the proof of~\Cref{prop:Ap-set}. The proof of the reverse inclusion is analogous to the proof of  
$\V (m\text{-focals}) \subseteq \GammaAp^{m,1}$ in that we construct a sequence in $\GammaAbarp^{m,1}$ 
that limits to any given point in $\V \left( \langle \text{specialized } \kfour \text{-focals} \rangle \right).$
However, instead of perturbing cameras, we perturb the associated world point, as illustrated in~\Cref{fig:sequence-q}.
We note that a similar argument appears in the proof of~\cite[Proposition 1]{DBLP:conf/iccv/TragerHP15}.

Let $\bar{\bA} = (\bar{A}_1,\ldots , \bar{A}_m)$ be a camera arrangement with distinct centers.
Suppose $\bar{\pp}$ lies in the vanishing locus of the $\bar{\bA}$-specialized $\kfour$-focals.
Then $\bar{\pp}$ also lies in the vanishing locus of all $m$-focals
and the corresponding specialized $m$-focal matrix 
is rank-deficient, so we may consider an element in its nullspace as in~\eqref{eq:null}.
At most one of $\lambda_1, \ldots, \lambda_m$ in that null vector can be zero since if, say, $\lambda_1 = \lambda_2 =0,$ then $\bar{q}$ is the center of both $\bar{A_1}$ and $\bar{A_2}.$
If all $\lambda_i$ are nonzero, then $\bar{\pp}$ is in the image of the imaging map \eqref{eq:imaging map} and thus $\bar{\pp}\in \GammaAbarp^{m,1}$.
Otherwise, we may assume without loss of generality that $\lambda_1 = 0,$ so that $\bar{q}$ is the center of $\bar{A_1}.$
We now write $\bar{\pp}$ as the limit of a sequence of points in $\GammaAbarp^{m,1}$
as $n\to \infty$, showing that $\bar{\pp}\in \GammaAbarp^{m,1}$.
Let $\bar{q}'$ be such that $\bar{A_1} \bar{q} ' \sim \bar{p}_1,$ and, working affinely, let $\bar{q}^{(n)} = \bar{q} + (1/n) \bar{q} '$.
Set $\bar{\pp}^{(n)} = (\bar{A}_1 \bar{q}^{(n)}, \ldots , \bar{A}_m \bar{q}^{(n)}).$
Passing to a subsequence if necessary, we may assume that $\bar{A}_i \bar{q}^{(n)} \ne 0$ for all $i\ge 1$ and all $n\ge 1.$
Observe for $i=1$ and all $n\ge 1$ that
\[
\bar{A}_1 \bar{q}^{(n)} \sim \bar{A}_1 \bar{q}'
  \sim \bar{p}_1
  \]
For $i\ge 2,$ in the limit as $n \to \infty$, 
 \[
  \bar{A}_i \bar{q}^{(n)} = \bar{A}_i \bar{q} + (1/n) \bar{A}_i \bar{q}' \to \bar{A}_i \bar{q}
  \sim \bar{p}_i.
  \]
We have exhibited $\bar{\pp}$ as the limit of points in $ \GammaAbarp^{m,1}.$
Since $\GammaAbarp^{m,1}$ is closed in the Euclidean topology, this 
implies that $\bar{\pp} \in \GammaAbarp^{m,1},$ proving the desired inclusion.
\end{proof}

\begin{figure}
    \centering
 \def\planeshade{black!5} 
 \begin{tikzpicture}
 \coordinate (q) at (-1.5, -1.5, 0);
 \draw[fill] (q) circle (0.05) node[below] {\scriptsize $\bar{q}$};
 \coordinate (r1) at (1,-2,-2);
 \coordinate (r2) at (1,-2,2);
 \coordinate (r3) at (1,2,2);
 \coordinate (r4) at (1,2,-2); 
 \draw[fill=\planeshade] (r1) -- (r2)  -- (r3) --  (r4) node[right] {} --  cycle;
 \coordinate (ri) at (1.12,-.53,0);
 \draw[fill] (ri) circle (0.05) node[below] {\scriptsize $\bar{A}_{i} \, \bar{q}\sim \bar{p}_i$};
 \def\sone{.32};
 \coordinate (rm1) at (1.12-2.62*\sone,-.53-.97*\sone,0);
 \draw[thick, dashed, -] (q) -- (rm1);
 \def\stwo{-.62};
 \coordinate (rm2) at (1.12-2.62*\stwo,-.53-.97*\stwo,0);
 \draw[thick, dashed, ->] (ri) -- (rm2);
 
 \coordinate (l1) at (-3.5, 0, 3);
 \coordinate (l2) at (-3.5,3,0);
 \coordinate (l3) at (-3.5,0,-3);
 \coordinate (l4) at (-3.5,-3,0);
 \draw[fill=\planeshade] (l1) -- (l2)  -- (l3) --  (l4) node[right] {} -- cycle;
 \coordinate (qprime) at (-5, 0, 0);
 \draw[fill] (qprime) circle (0.05) node[below] {\scriptsize $\bar{q} '$};
 
 \def\sthree{.37};
 \coordinate (lm1) at (-5 + 3.5*\sthree, -1.5*\sthree, 0);
 \draw[thick, dashed, -] (qprime) -- (lm1);
 \draw[fill] (lm1) circle (0.05) node[below] {\scriptsize $ \bar{p}_1$};
 \def\sfour{.62};
 \coordinate (lm2) at (-5 + 3.5*\sfour, -1.5*\sfour, 0);
 \draw[thick, dashed, ->] (lm2) -- (q);
 \def\sfive{.81};
 \coordinate (qn) at (-5 + 3.5*\sfive, -1.5*\sfive, 0);
 \draw[fill] (qn) circle (0.05) node[below] {\scriptsize $ \bar{q}^{(n)}$};
 \def\ssix{.481};
 \coordinate (qn1) at (-2.165 + 4.9094* \ssix , -1.215 + 1.2864 * \ssix , 0);
 \draw[dashed, thick, -] (qn) -- (qn1);
 \def\sseven{.681};
 \coordinate (qn2) at (-2.165 + 4.9094* \sseven , -1.215 + 1.2864 * \sseven , 0);
  \def\seight{1};
 \coordinate (qn3) at (-2.165 + 4.9094* \seight , -1.215 + 1.2864 * \seight , 0);
 \draw[dashed, thick, -] (qn2) -- (qn3);
  \draw[fill] (qn2) circle (0.05) node[above] {\scriptsize $\bar{A}_{i} \, \bar{q}^{(n)}$};
 \end{tikzpicture}
    \caption{Proof of~\Cref{prop:multiview-focals}; the $\lambda_1 = 0$ case.}
    \label{fig:sequence-q}
\end{figure} 

\begin{proof}[Proof of~\Cref{thm:AST-UGB}]
Let $<$ be any product order on $\CC[\bA, \pp]$ with $\bA < \pp,$ formed from product orders $<_\pp  $ on $\CC [\pp]$ and $<_\bA$ on $\CC[\bA].$
We note the following identity:
\begin{equation}
\label{eq:elimination-order}
in_< \left( \displaystyle\sum_{\pp^{\alpha_1} \, <_\pp \,  \cdots \, <_\pp \, \pp^{\alpha_k} } g_{\alpha_i } (\bA) \pp^{\alpha_i}  \right) = in_{<_\bA} (g_{\alpha_k}(\bA)) \pp^{\alpha_k} .
\end{equation}
Using this identity and~\Cref{thm:Ap-GB}, application of Buchberger's S-pair criterion~\cite[\S 2.6]{CLO15} shows that the $\bar{\bA}$-specialized $\kfour$-focals form a Gr\"{o}bner basis with respect to $<_\pp $ as long as none of their leading coefficients $g_{\alpha_k} (\bar{\bA})$ vanish.
Each of these leading coefficients is a $4\times 4$ minor of $\stacked{\bar{A}_1}{\bar{A}_m}$ and by 
minor genericity, they do not vanish. 
Since $<_\pp$ is an arbitrary monomial order on $\CC [\pp ]$, the specialized $\kfour$-focals form a universal Gr\"{o}bner basis.
To show they generate $\Abarp^{m,1}$, apply the recognition criterion of~\Cref{prop:prop-recognition}.
The set-theoretic statement is just~\Cref{prop:multiview-focals}, and the remaining conditions follow as in~\Cref{thm:Ap-GB}.
\end{proof}

Combining the last example with our previous results allows us to recover the main ideal-theoretic result of~\cite{APT21}.

\begin{proof}[Proof of~\Cref{cor:2-3-focal}]
When $\bar{\bA}$ is minor generic, this follows by applying~\Cref{thm:AST-UGB} and specializing~\eqref{eq:sat-23-focal} for every $4$-focal $f.$
For $\bar{\bA}$ with pairwise distinct centers,~\Cref{thm:orbits} shows that there exists $h=(H_1,\ldots,H_m) \in 
(\PGL_3)^m$ such that $$h\bar{\bA} = (H_1 \bar{A}_1, \ldots, H_m \bar{A}_m)$$ is minor generic. Consider the automorphism of $\ZZ^m$-graded rings defined as 
\begin{align*}
L_{h} : \CC [\pp ] &\to \CC [\pp] \\
(p_1, \ldots , p_m ) &\mapsto (H_1 p_1, \ldots , H_m p_m).
\end{align*}
This allows us to reduce to the minor generic case:
\begin{align*}
\langle \kthree \text{-focals} |_{\bar{\bA}} \rangle &= L_h \left( \langle \kthree \text{-focals} |_{h\bar{\bA}} \rangle  \right) \tag{\cite[Lemma 2.3,  3.6]{APT21}}\\ 
&= L_h \left( \iideal{h\bar{\bA}, \pp}^{m,1}  \right) \tag{$h \bar{\bA}$ is minor generic}\\
&= \Abarp^{m,1}.
\end{align*}
In the above calculation, we note that the results used from~\cite{APT21} are proven by Cauchy-Binet type arguments analogous to the proof of~\Cref{thm:orbits}.
The final equality follows directly from the definition of $\Abarp^{m,1}.$
Since $\bar{\bA}$ was an arbitrary camera arrangement with pairwise distinct centers, the proof is complete.
\end{proof}

\section{$\Aqp^{m,1}$ : Constraints on cameras, world points, and image points for $n=1$}\label{sec:Aqp}

We now move on to $\Aqp^{m,1}$ for $m$ cameras and 
one world point. 
Recall that this is the vanishing ideal of 
the image formation correspondence $\GammaAqp^{m,1}$ \eqref{eq:GammaAqp}, and the second bundle adjustment ideal in~\Cref{fig:ideals}.

For fixed $i,$ the condition $A_i q \sim p_{i}$ is equivalent to the vanishing of all $2\times 2$ minors of the matrix $\left(\BM A_i q & p_{i } \EM\right)$. 
Thus, we may consider the ideal generated by all such minors, which we denote as:
\begin{equation}
\label{eq:minor-ideal}
\minorideal = \sum_{i=1}^m \textup{minors}\left(2, \left(\BM A_i q & p_i \EM\right)\right).
\end{equation}

Each summand, $\textup{minors}\left(2, \left(\BM A_i q & p_i \EM\right)\right)$, appearing in~\eqref{eq:minor-ideal} looks similar to a well-studied example in combinatorial commutative algebra, namely, the $2\times 2$ minors of a $3\times 2$ matrix of indeterminates, which form a quadratic universal Gr\"{o}bner basis by~\cite[Example 1.4]{Stu96}.
However, this analogy cannot be pursued too closely here, as the minors in~\eqref{eq:minor-ideal} are \emph{cubics} and the matrix entries are not indeterminates.

The ideal $\minorideal $ is a natural candidate for $\Aqp^{m,1}.$
However, in general $\minorideal \subsetneq \Aqp^{m,1}$.

\begin{example}
\label{ex:Aqp-1}
For $m=2$ cameras, we can easily check in \texttt{Macaulay2} that 
\[
\Aqp^{2,1} = M^{2,1}_{\bA , \qq , \pp} + \Ap^{2,1}
\supsetneq 
M^{2,1}_{\bA , \qq , \pp} = \Aqp^{2,1} \cap \frakm_\qq .
\]
\end{example}

In spite of~\Cref{ex:Aqp-1}, the ideals $\Aqp^{m,1}$ and $\minorideal$ are closely related. 
Our main result of this section, \Cref{thm:Aqp-master-n1} below,  establishes that $\Aqp^{m,1} = \minorideal + \Ap^{m,1}$.

\begin{theorem}The vanishing ideal $\Aqp^{m,1} $ of the 
variety $\GammaAqp^{m,1}$ may be described as follows:
  \label{thm:Aqp-master-n1}

For any $1\leq i \leq 4$,
\begin{equation}\label{eq:Aqp-sat}
\Aqp^{m,1} = \minorideal : q [i],
\end{equation}
 and
\begin{equation}\label{eq:Aqp-gen}
\Aqp^{m,1} = \minorideal + \Ap^{m,1}.
\end{equation}
Thus, $\Aqp^{m,1}$ is generated by the $2\times 2$ minors generating $\minorideal$ and the $\kfour$-focals.
  \end{theorem}

Once again, we use the recognition criterion of~\Cref{prop:prop-recognition} to characterize $\Aqp^{m,1}.$
The set-theoretic equalities
\begin{equation}\label{eq:Aqp-set}
\GammaAqp^{m,1} = \V (\minorideal) = \V (\minorideal + \Ap^{m,1})
\end{equation}
follow along the same lines as in~\Cref{prop:Ap-set}.
To verify the remaining conditions in~\Cref{prop:prop-recognition}, we identify explicit Gr\"{o}bner bases with squarefree initial ideals.
At the core of our arguments is the following technical result, whose proof, unlike~\Cref{thm:Ap-GB}, is computer-assisted.

\begin{proposition}
  \label{prop:MAqp-GB}
  There exists a set $G_\minorideal$, which is a Gr\"{o}bner basis for $\minorideal$ with respect to all $12$ Lex or GRevLex orders $<$ satisfying
  \[
  A_m [3,4] < \ldots < A_1 [1,1] , \phantom{ff} 
  q[4] < \ldots < q[1] , \phantom{ff} 
  p_m [3] < \cdots < p_1 [1],
  \]
  where $<$ refines the column-major order on each matrix $A_i, q, p_i$ and the groups $\bA, \qq, \pp$ may be permuted in any of the six possible ways.
\end{proposition}

Elements of the set $G_\minorideal$ are listed explicitly in~\Cref{sec:GB-appendix} and their 
  degrees are tabulated in \Cref{fig:Betti table}.

\begin{proof}
To begin, we record several important properties of $G_\minorideal .$
\begin{itemize}
\item[\textbf{P1}] For $m\ge 4,$ 
\[
G_\minorideal = \displaystyle\bigcup_{\sigma \in \binom{[m]}{4}} G_\minorideal^\sigma ,
\]
where $G_\minorideal^\sigma = G_\minorideal \cap \CC [A_{\sigma_1}, A_{\sigma_2}, 
A_{\sigma_3}, A_{\sigma_4}, q, p_{\sigma_1}, p_{\sigma_2}, p_{\sigma_3}, p_{\sigma_4} ].$ 
\item[\textbf{P2}] $G_\minorideal$ contains all elements of the form $q[4]$ times a $2,3,$ or $4$-focal.
\item[\textbf{P3}] $\langle G_\minorideal \rangle = \minorideal $
\item[\textbf{P4}] Aside from the $q[4]$-bumped focals, no other elements of $G_\minorideal $ are divisible by any variable in $\CC[\bA, \qq, \pp ].$
\item[\textbf{P5}] The lead terms of all $g\in G_\minorideal$ are squarefree.
\end{itemize}

\textbf{P1} and \textbf{P2} follow upon inspecting the polynomials listed in~\Cref{sec:GB-appendix}, each of which depends on at most $4$ cameras and their corresponding image points.

For \textbf{P3}, the inclusion $\minorideal \subset \langle G_\minorideal \rangle $ holds since $G_\minorideal$ contains all the $3m$ cubic $2 \times 2$ minors generating $\minorideal .$
The reverse inclusion and the assertions \textbf{P4} and \textbf{P5} may all be verified for $m=1,\ldots , 4$ using \texttt{Macaulay2}, and for $m>4$ using \textbf{P1}. Here we are using the fact that 
$\minorideal$ is the sum of all minor ideals $M^{|\sigma|,1}_{\bA,\qq,\pp}$ as $\sigma$ 
varies over all 
$4$ element subsets of $[m]$. 

Having established \textbf{P1}--\textbf{P5}, \texttt{Macaulay2} verifies that $G_\minorideal$ is a Gr\"{o}bner basis for all $12$ term orders for each of the ``base cases" $m=1,\ldots , 8$ .
For $m>8,$ we use Buchberger's S-pair criterion: $G_\minorideal$ is a Gr\"{o}bner basis 
for the monomial order $<$ if and only if for any $f,g \in G_\minorideal ,$ the $S$-pair \[
S(f, g) = \lcm(in_< (f), in_< (g))  (f/in_<(f) - g / in_< (g))
\]
has zero remainder after division by $G_\minorideal .$
Consider any $f,g\in G_\minorideal$. By \textbf{P1},
we have $f\in G_\minorideal^{\sigma}$
and $g\in G_\minorideal^{\tau}$
for some $\sigma,\tau \in \binom{[m]}{4}.$
Thus
\[
S(f,g) \in \langle G_{\minorideal}^{\sigma} \cup G_{\minorideal}^{\tau} \rangle . 
\]
This S-pair involves at most $8$ cameras, so its remainder upon division by $G_\minorideal $ is zero by the base cases.
\end{proof}

In contrast to the set of $\kfour$-focals that form a Gr\"obner basis for 
$\Ap^{m,1},$ the Gr\"{o}bner basis $G_\minorideal$ is quite mysterious.
We list three additional properties which may provide further clues to its structure.

\begin{itemize}
\item[\textbf{P6}] The elements of $G_{\minorideal}$ are linear in the variable group $\qq,$ and may be partitioned according to their support as follows:
\begin{itemize}
\item[] Degree-$3$ elements have $\qq$-support $\{ q[1], q[2], q[3], q[4] \}$.
\item[] Degree-$4,5$ elements have $\qq$-support $\{ q[2], q[3], q[4] \}$.
\item[] Degree-$6,7$ elements which are not bumped $2$-focals have $\qq$-support $\{ q[3], q[4]\}$.
\item[] Bumped $\kfour$-focals have $\qq$-support $\{ q[4] \}.$
\end{itemize}
\item[\textbf{P7}] Viewed as polynomials in $\qq$ and $\pp ,$ the leading coefficient of any element of $G_\minorideal$ with respect to either of the induced Lex orders on $\CC[\bA] [\qq, \pp]$ is always a $k \times k$ minor of $\left(\BM {A}_1^\top & \cdots & {A}_m^\top \EM\right)$ for some $1\le k \le 4.$
\item[\textbf{P8}] With respect to the $12$ monomial orders in~\Cref{prop:MAqp-GB}, there are $4$ distinct initial ideals $in_< (\minorideal ),$ realized by letting $<$ be one of the two GRevLex orders with $\qq < \bA, \pp$ or one of the two Lex orders with $\qq > \bA , \pp .$
\end{itemize}

By~\Cref{prop:MAqp-GB}, it suffices to check these properties for $m=4$ cameras.
Property \textbf{P7} justifies the notion of ultra minor genericity in~\Cref{def:ultra-minor-generic}, which we make use of in~\Cref{sec:Abarqp}.

\begin{definition} \label{def:Aqp-GB}
Let $G_{\Aqp^{m,1}}$ be the set consisting of all $\kfour$-focals together with the elements of $G_\minorideal$ which are not bumped focals.
\end{definition}

Alternatively, $G_{\Aqp^{m,1}}$ consists of the elements of $G_\minorideal$ after bumping down
the elements of the form $q[4] \cdot \textup{$k$-focal}$ where $k=2,3,4$. 
See \Cref{fig:Betti table} for the degrees of elements in $G_{\Aqp^{m,1}}$. 

\begin{figure}
\begin{center}
\begin{tabular}{c|c|c}
\textup{degree} & $G_\minorideal$ & $G_{\Aqp^{m,1}}$\\ \hline 
$3$  & $3m$ & $3m$ \\
$4$  & $m$ & $m$ \\
$5$  & $9 \binom{m}{2}$ &  $9 \binom{m}{2}$\\
$6$  & $6\binom{m}{2}$ & $7\binom{m}{2}$\\
$7$  & $\binom{m}{2} + 27 \binom{m}{3}$ & $54 \binom{m}{3}$\\
$8$  & $27 \binom{m}{3}$ & $81 \binom{m}{4}$\\
$9$  & $81 \binom{m}{4}$ & $0$ 
\end{tabular}
\end{center} 
\caption{The number of generators of each degree in $G_\minorideal $ and $G_{\Aqp^{m,1}}$.
\label{fig:Betti table}}
\end{figure}

Our notation $G_{\Aqp^{m,1}}$ foreshadows the not-yet-proven fact that this set 
is a Gr\"{o}bner basis of $\Aqp^{m,1}.$ 
This will be established in the proof of of~\Cref{thm:Aqp-master-n1}. First we write down some 
corollaries to \Cref{prop:MAqp-GB}.

\begin{corollary}
\label{cor:Aqp-GB}
\begin{enumerate}
    \item $G_{\Aqp^{m,1}}$ is a Gr\"{o}bner basis for the ideal 
     \begin{equation} \label{eq:MAqp-saturated-q}
        \minorideal + \Ap^{m,1} = \minorideal :  q[4]
     \end{equation}
    with respect to any of the 12 monomial orders defined in~\Cref{prop:MAqp-GB}.
    \item Moreover,
\begin{equation}
\label{eq:MAqp-saturated-Ap}
\minorideal + \Ap^{m,1} = \left( \minorideal + \Ap^{m,1} \right) : A_m [3,4] = \left( \minorideal + \Ap^{m,1} \right) : p_m [3] .
\end{equation}
\end{enumerate}
\end{corollary}

\begin{proof}
Let $<$ be one of the two GRevLex orders appearing in property \textbf{P8} of $G_\minorideal .$ 
Since $G_\minorideal$ is a Gr\"obner basis for $\minorideal$ with respect to $<$ by \Cref{prop:MAqp-GB}, combining the definition 
of $G_{\Aqp^{m,1}}$ and \Cref{lem:grevlex} says that $G_{\Aqp^{m,1}}$ is a Gr\"obner basis 
of $\minorideal :  q[4]$ with respect to $<$. On the other hand, 
the ideal generated by $G_{\Aqp^{m,1}}$ is precisely $\minorideal + \Ap^{m,1}$. Indeed, 
by the construction in \Cref{def:Aqp-GB}, $\langle G_{\Aqp^{m,1}} \rangle$ contains 
$\minorideal + \Ap^{m,1}$, while every element of $G_{\Aqp^{m,1}}$ lies in $\minorideal + \Ap^{m,1}$. 
Therefore, we get that $ \minorideal + \Ap^{m,1} = \minorideal :  q[4]$. 

Now suppose $<$ is one of the two Lex orders appearing in \textbf{P8}.
Using \textbf{P6}, the elements of $G_\minorideal $ are well-supported with respect to $\qq ,$ and hence we may apply~\Cref{lem:colon-wellsupported} to conclude that $G_{\Aqp^{m,1}}$ is a Gr\"obner basis 
of $\minorideal :  q[4]$ with respect to $<$.

Next we argue that $G_{\Aqp^{m,1}}$ is a Gr\"obner basis for $\minorideal + \Ap^{m,1}$ with respect to any of the remaining $8$ orders considered in \Cref{prop:MAqp-GB}. 
Letting $<'$ be such an order, \textbf{P8} lets us choose $<$ from the $4$ previously considered orders such that
\begin{equation}\label{eq:equal-inI}
in_{<'} (\minorideal) 
= in_{<} (\minorideal ).
\end{equation}
We may apply either \Cref{lem:grevlex} when $<$ is a GRevLex order or \Cref{lem:colon-wellsupported} when $<$ is a Lex order to the monomial ideal $in_< (\minorideal ).$ 
In either case this gives 
\begin{align} \label{eq:initial equality}
(in_< (\minorideal )) : q[4]  = \langle in_< (G_{\Aqp^{m,1}}) \rangle,
\end{align}
where $in_< (G_{\Aqp^{m,1}})$ is the set of leading terms of elements in $G_{\Aqp^{m,1}}$. 
Thus,
\begin{align*}
 in_{<'} \langle G_{\Aqp^{m,1}} \rangle & = in_{<'}(\minorideal : q[4]) \tag{$\langle G_{\Aqp^{m,1}} \rangle = \minorideal : q[4]$}\\ 
&\subset (in_{<'} (\minorideal)) : q[4] \\
&= (in_{<} (\minorideal )) :  q[4] \tag{Equation~\eqref{eq:equal-inI}}\\
&= \langle in_{<} (G_{\Aqp^{m,1}}) \rangle \tag{Equation~\eqref{eq:initial equality}}\\
&=  in_{<} \langle G_{\Aqp^{m,1}} \rangle \tag{$G_{\Aqp^{m,1}}$ is a Gr\"{o}bner basis with respect to $<$.}
\end{align*}    
The inclusion $\subset$ follows from the definition of an ideal quotient.
Since distinct initial ideals are incomparable with respect to inclusion, we conclude
\[
in_{<'} (G_\Aqp^{m,1}) = in_{<} (G_\Aqp^{m,1}),
\]
showing that $G_\Aqp^{m,1}$ is a Gr\"{o}bner basis with respect to all $12$ term orders.

Finally, we establish~\eqref{eq:MAqp-saturated-Ap}. Note that since no polynomial in 
$G_{\Aqp^{m,1}}$ has $A_m[3,4]$, $p_m[3]$ or $q[4]$ as a factor,  by \Cref{lem:grevlex}, $G_{\Aqp^{m,1}}$ is a 
GRevLex Gr\"obner basis of $(\minorideal + \Ap^{m,1}) : A_m[3,4]$,  $(\minorideal + \Ap^{m,1}) : q[4]$ and $(\minorideal + \Ap^{m,1}) : p_m[3]$.
Hence, all three of these ideal quotients equal the ideal generated by $G_\Aqp^{m,1}$ which 
is $ \minorideal + \Ap^{m,1}$. 
\end{proof}

To complete the proof of~\Cref{thm:Aqp-master-n1}, we need a basic lemma on the quotient by some variable of a polynomial ideal that is invariant under the action of a permutation group.
\begin{lemma}
\label{lem:symmetric-saturation}
Let $P \subset S_k$ be a permutation group with its standard action on $\CC [x_1, \ldots , x_k].$
Suppose $I \subset \CC [x_1, \ldots , x_k]$ is an ideal which is $P$-invariant, meaning, 
$P \cdot I = I$. 
Then
\[
I = I : x_i 
\phantom{ff}
\Rightarrow 
\phantom{ff}
I = I :  x_{\rho (i)} 
\phantom{ff}
\forall \rho \in P.
\]
\end{lemma}
\begin{proof}
Suppose $f\in I : x_{\rho (i)}  $. 
Equivalently,
\[
(\rho \cdot x_i) f = x_{\rho( i)} f \in I.
\]
We need to show $f\in I$ as well.
Note that 
\[
\rho^{-1} \cdot (x_{\rho (i)} f) = x_i (\rho^{-1} \cdot f)
\in P \cdot I = I.
\]
This gives 
\[
\rho^{-1} \cdot f \in I : x_i = I.
\]
Thus,
\[
f = \rho \cdot (\rho^{-1} \cdot f) 
\in P \cdot I = I.
\]
\end{proof}

\begin{corollary}
\label{cor:Aqp-sat}
The following identities hold for any $i\in [m], j \in [3], k \in [4]$.
\begin{align*}
\minorideal : q[4]  &= \minorideal : q[k],  \\
\minorideal : A_m [3,4] &= \minorideal : A_i [j,k],\\
\minorideal : p_m [3] &= \minorideal : p_i [j].
\end{align*}
\end{corollary}
\begin{proof}
We deduce each identity from~\Cref{lem:symmetric-saturation} by producing a suitable permutation $\rho \in P$, where the group $P \subset S_{4 + 15m}$ is generated by all
permutations taking one of three forms:

\begin{enumerate}
\item Camera permutations, given by some $\rho \in S_m$ which sends $A_i$ to $A_{\rho (i)}$
and $p_i$ to $p_{\rho (i)}$.
\item World coordinate permutations, given by some $\rho \in S_4$ which sends $q[k]$ to $q[\rho(k)]$ and sends $A_i[:,k]$ to $A_i[:,\rho (k)].$
\item Image coordinate permutations, given by $\rho \in S_3$ which sends the $j$th coordinate of $p_i [j]$ to $p_i [\rho (j)]$ and $A_i [j, :]$ to $A_i [\rho (j) , :].$ 
\end{enumerate}
Note that $\minorideal$ is invariant 
under the action of $P$.
To deduce the first identity in the statement of the Corollary, apply a permutation of the form 2 above.
For the second identity, apply 2, then 3, then 1.
For the third identity, apply 2, then 1.
\end{proof}

\begin{proof}[Proof of~\Cref{thm:Aqp-master-n1}]
By \Cref{cor:Aqp-GB}, $G_{\Aqp^{m,1}}$ is a Gr\"{o}bner basis for $\minorideal + \Ap^{m,1}$ whose 
initial monomials are squarefree. Therefore, $\minorideal + \Ap^{m,1}$ is a radical ideal by 
\Cref{prop:radical}.
By~\Cref{cor:Aqp-GB} and~\Cref{cor:Aqp-sat}, this ideal is saturated with respect to all variables.
This gives us both~\eqref{eq:Aqp-sat} and that $$\minorideal + \Ap^{m,1}$$ is saturated with respect to the irrelevant ideal.
Equation~\eqref{eq:Aqp-set} gives the set-theoretic statement $\GammaAqp^{m,1} = \V  (\minorideal) = \V (\minorideal + \Ap^{m,1})$.
Thus, \eqref{eq:Aqp-gen}, which states that \[
\minorideal + \Ap^{m,1} = \Aqp^{m,1},
\]
now follows from~\Cref{prop:prop-recognition}.
\end{proof}

\section{$\Abarqp^{m,1}$: Completing the square for $n=1$}\label{sec:Abarqp}
Just as our results on the ideal $\Ap^{m,1}$ allow us to recover known results about the multiview ideal $\Abarp^{m,1} ,$ we may combine the results of the previous section with specialization arguments to study $\Abarqp^{m,1} .$ 
This ideal is an analogue of the multiview ideal that does not eliminate the world point variables $\qq$. 
It is the final ideal in the square seen in \Cref{fig:ideals}.
The next two theorems characterize $\Abarqp^{m,1}$ Gr\"{o}bner-theoretically and ideal-theoretically.

\begin{theorem}
\label{thm:Abarqp-gb}
 Specializing the Gr\"{o}bner basis $G_{\Aqp^{m,1}}$ to an ultra minor generic camera arrangement $\bar{\bA}$ yields a (non-reduced) Gr\"{o}bner basis for $\Abarqp^{m,1}$ with respect to the restrictions of Lex or GRevLex orders to $\CC [\qq, \pp]$ such that $\bA> \qq, \pp .$
\end{theorem}
\begin{proof}
The strategy is much like the proof of~\Cref{thm:AST-UGB}.
Property \textbf{P7} of $G_\Aqp^{m,1}$ implies that the leading coefficients of the elements of $G_\Aqp^{m,1}$ in $\CC[\bA]$ are nonzero after specialization to an ultra minor generic camera arrangement $\bar{\bA}$.
Since the Lex orders are product orders,
the specialized Gr\"{o}bner basis elements $G_\Abarqp^{m,1}$ form a Gr\"{o}bner basis for both Lex orders in the statement of the theorem.
Moreover, the leading terms of $G_\Abarqp^{m,1}$ with respect to either Lex order are the same for the corresponding GRevLex order. Comparing leading terms as in \Cref{cor:Aqp-GB} proves that $G_\Abarqp^{m,1}$ forms a Gr\"{o}bner basis for the ideal it generates.
Finally, to prove that this ideal is $\Abarqp^{m,1}$, apply the recognition criterion of~\Cref{prop:prop-recognition} in a manner analogous to the proof of~\Cref{thm:Aqp-master-n1}.
\end{proof}

\begin{theorem}\label{thm:Abarqp-ideal}
If $\bar{\bA}$ has pairwise distinct centers, then specializing the $2\times 2$ minors of $\BM( A_i q & p_i\EM)$  and $\kthree$-focals yields generators for $\Abarqp^{m,1}$. In particular,
    \begin{equation}
    \label{eq:Abarqp-sum}
    \Abarqp^{m,1}  = M^{m,1}_{\bar{\bA}, \qq, \pp} + \Abarp^{m,1}.
    \end{equation}
\end{theorem}

\begin{proof}
A group element $h = (H, H_1 , \ldots , H_m) \in \PGL_4 \times (\PGL_3)^m$ induces a linear automorphism of the $\ZZ^{m+1}$-graded ring $\CC [\qq, \pp]:$
\begin{align*}
L_h : \CC [\qq, \pp] \to \CC [\qq, \pp]\\
q \mapsto H q, \quad
p_i \mapsto H_i p_i.
\end{align*}
We have already shown in~\Cref{thm:Abarqp-gb} that~\eqref{eq:Abarqp-sum} holds when $\bar{\bA}$ is 
ultra minor generic.
To extend this to $\bar{\bA}$ with pairwise distinct centers, we apply~\Cref{thm:orbits} which reduces the proof to checking that the ideals $M^{m,1}_{\bar{\bA}, \qq, \pp}$ and $\Abarp^{m,1}$ are preserved under coordinate change, ie.
\begin{equation*}
M^{m,1}_{\bar{\bA}, \qq, \pp} = L_h \left(M_{h\bar{\bA}, \qq, \pp}^{m,1}\right),\quad
\Abarp^{m,1} = L_h \left( \iideal{h \bar{\bA}, \pp}^{m,1}\right).
\end{equation*}
The second equality appeared already in the proof of~\Cref{cor:2-3-focal}.
The first may be proven along similar lines: we calculate
\begin{align*}
\left(\BM \bar{A}_i q & p_i \EM\right) &= \left(\BM H_i^{-1} (H_i A_i H^{-1}) H q & H_i^{-1} (H_i p_i)  \EM\right)\\
&= H_i^{-1} \cdot \left(\BM (h \bar{\bA})_i \cdot L_h (q) & L_h (p_i)  
 \EM\right).
\end{align*}
A Cauchy-Binet argument analogous to that given in~\Cref{sec:proofs appendix} shows that every $2\times 2$ minor of the matrix on the left-hand side is a linear combination of $2\times 2$ minors of the matrix on the right.
This gives the inclusion $\Abarp^{m,1} \subset L_h \left( \iideal{h \bar{\bA}, \pp}^{m,1} \right)$.
The reverse inclusion is similar.
\end{proof}

In closing, we note that the ideals appearing in this section and the last provide another route to understanding the ideals obtained by eliminating $\qq$ from them.
For instance, we could use the results of these sections to give a different proof of~\Cref{cor:2-3-focal}, or the ideal-theoretic part of~\Cref{thm:Ap-GB}.

\section{$\Aqp^{m,n}, \Ap^{m,n}, \Abarqp^{m,n}, \Abarp^{m,n}$: The square for $n\ge 1$}
\label{sec:n-points}

So far, we have characterized the ideals in~\Cref{fig:ideals} for the case of $n=1$ world point and $m$ cameras. 
In this section, we establish results for these ideals when $n\ge1.$

For each $k=1, \ldots , n,$ denote as follows the elimination ideals 
involving the $k$th world point:
\begin{equation}\label{eq:restricted-ideals}
\begin{split}
\Aqp[k]^{m,1} &= \Aqp^{m,n} \cap \CC [\bA, \qq_k, \pp_{1 k} \ldots , \pp_{m k}],\\
\Ap[k]^{m,1} &=  \Ap^{m,n} \cap \CC [\bA, \pp_{1 k} \ldots , \pp_{m k}],\\
\Abarqp[k]^{m,1} &= \Abarqp^{m,n} \cap \CC [\qq_k, \pp_{1 k} \ldots , \pp_{m k}],\\
\Abarp[k]^{m,1} &= \Abarp^{m,n} \cap \CC [\pp_{1 k} \ldots , \pp_{m k}].
\end{split}
\end{equation}
One naturally hopes that summing the appropriate extensions of the ideals in~\eqref{eq:restricted-ideals}, over $k=1, \ldots , n$, would give us the ideals of~\Cref{fig:ideals} for arbitrary $n$. If this were true, our results about explicit generators for these ideals in the $n=1$ case would generalize effortlessly.
For the triangulation ideals $\Abarqp^{m,n}$ and $\Abarp^{m,n}$, this ideal scenario turns out to be true. 
The essential observation is that $\GammaAbarqp^{m,n}$ and $\GammaAbarp^{m,n}$ are both direct products of the restricted varieties:
\begin{align*}
\GammaAbarqp^{m,n} &= \V (\Abarqp[1]^{m,1} ) \times \cdots \times \V (\Abarqp[n]^{m,1} ),\\
\GammaAbarp^{m,n} &= \V (\Abarp[1]^{m,1} ) \times \cdots \times \V (\Abarp[n]^{m,1} ).
\end{align*}

Let $\CC [\qq , \pp ] \, \Abarqp[k]^{m,1}$ be the extension of  
$\Abarqp[k]^{m,1}$ in the ring $\CC[\qq,\pp]$, i.e., the ideal generated by the 
generators of $\Abarqp[k]^{m,1}$ in $\CC[\qq,\pp]$. 
The first result of this section determines explicit generators for the triangulation ideals $\Abarqp^{m,n}$ and $\Abarp^{m,n}$ whenever $\bar{\bA}$ has pairwise distinct centers.

\begin{theorem}[Triangulation ideals when $n\ge 1$]\label{thm:sum-ideal}
For $\bar{\bA}$ with pairwise distinct centers and any $n\ge 1$, the triangulation ideals 
are obtained by summing the extensions of individual ideals in the following sense:
\begin{equation}
\label{eq:sum-Abarqp}
\Abarqp^{m,n} = \displaystyle\sum_{k=1}^n   \CC [\qq , \pp ] \, \Abarqp[k]^{m,1} ,
\end{equation}
\begin{equation}
\label{eq:sum-Abarp}
\Abarp^{m,n} = \displaystyle\sum_{k=1}^n \CC [\pp ]\, \Abarp[k]^{m,1} .
\end{equation}
\end{theorem}

Under appropriate genericity assumptions on $\bar{\bA}$, the Gr\"{o}bner bases studied in~\Cref{sec:Abarp} and~\Cref{sec:Abarqp} may be combined to give Gr\"{o}bner bases for any $n \ge 1$.
A precise statement is the following result.

\begin{theorem}[Gr\"obner bases for triangulation ideals when $n\ge 1$]\label{thm:sum-ideal-GB}
Suppose $G_1, \ldots , G_n$ are Gr\"{o}bner bases for the restricted ideal $\Abarqp[k]^{m,1}$ (or $\Abarp[k]^{m,1}$) for $k=1,\ldots , n$, 
with respect to the term orders $<_1, \ldots , <_n$ allowed in  \Cref{thm:Abarqp-gb} (or \Cref{thm:AST-UGB}). 
If $<$ is any term order on $\CC [\qq, \pp]$ (or $\CC [\pp]$) refining $<_1, \ldots , <_n,$ then $G = G_1\cup \cdots \cup G_n$ is a Gr\"{o}bner basis with respect to $<$ for the triangulation ideal $\Abarqp^{m,n}$ (or $\Abarp^{m,n}$).
\end{theorem}

In particular, if $\bar{\bA}$ is a minor generic arrangement, then the set of specialized $\kfour$-focals form a universal Gr\"{o}bner basis for $\Abarp^{m,n}.$
Likewise, if $\bar{\bA}$ is ultra minor generic, then the union of the Gr\"{o}bner bases $G_{\Abarqp[k]^{m,1}}$ described in~\Cref{thm:Abarqp-gb} is a Gr\"{o}bner basis for $\Abarqp^{m,n}$ with respect to a suitable refinement of the GRevLex or Lex orders used 
in \Cref{thm:Abarqp-gb}.

In the case of the two bundle adjustment ideals in \Cref{fig:ideals} we can write down an 
ideal-theoretic result using our results for $n=1$. These ideals may be obtained by summing the corresponding extended ideals for each individual world point, and then saturating by a polynomial defining an appropriate locus of non-generic cameras.
\begin{theorem}\label{thm:saturation}
For any $n\ge 1$ we have
\begin{equation}
\label{eq:sum-Aqp}
\Aqp^{m,n} = \left( \displaystyle\sum_{k=1}^n \CC [\bA, \qq, \pp] \, \Aqp[k]^{m,1}\right) : s_{\text{\rm ultra}}^\infty ,
\end{equation}
\begin{equation}
\label{eq:sum-Ap}
\Ap^{m,n} = \left( \displaystyle\sum_{k=1}^n \CC [\bA , \pp ] \, \Ap[k]^{m,1}\right) : s^\infty ,
\end{equation}
where $s$ is the product of all $4\times 4$ minors of $\left( \BM A_1^\top & \cdots & A_m^\top \EM \right),$ and $s_{\text{\rm ultra}}$ is the product of all minors (of any size) of $\left( \BM A_1^\top & \cdots & A_m^\top \EM \right)$.
\end{theorem}

The following is an immediate corollary of \Cref{thm:saturation}.
\begin{corollary} \label{cor:saturation}
Set-theoretically,
\begin{align}
\label{eq:gammas-complement}
\GammaAqp^{m,n} = & \cl \left( \V \left( \displaystyle\sum_{k=1}^n \CC[\bA,\qq,\pp] \Aqp[k]^{m,1}\right) \setminus \V (s_{\text{ultra}}) \right), \textup{ and }\\
\GammaAp^{m,n} = &\cl \left(\V \left( \displaystyle\sum_{k=1}^n \CC[\bA,\pp]\Ap[k]^{m,1}\right) \setminus \V (s) \right).
\end{align}
\end{corollary}

However, we can also prove a slightly stronger set-theoretic statement.

\begin{theorem}\label{thm:ba-set-theoretic}
Consider the subvariety $C_{\bA, \qq, \pp}^{m,n} \subset \GammaAqp^{m,n}$ consisting of all tuples $(\bar{\bA}, \bar{\qq}, \bar{\pp})$ where $\bar{\bA}$ is such that the kernels of some pair of camera matrices intersect nontrivially.
Let $C_{\bA , \pp}^{m,n}$ denote its projection into $\GammaAp^{m,n}.$ 
Then
\begin{align}
\GammaAqp^{m,n} = & \cl \left(\V \left( \displaystyle\sum_{k=1}^n \CC[\bA,\qq,\pp] \Aqp[k]^{m,1}\right) \setminus C_{\bA , \qq, \pp}^{m,n} \right),\\
\GammaAp^{m,n} = & \cl \left(\V \left( \displaystyle\sum_{k=1}^n \CC[\bA,\pp] \Ap[k]^{m,1}\right) \setminus C_{\bA, \pp}^{m,n} \right).
\end{align}
\end{theorem}
We note that~\Cref{thm:ba-set-theoretic} automatically implies~\Cref{cor:saturation}, since every (ultra) minor generic camera arrangement has pairwise distinct centers. 
For instance, 
since $C_{\bA , \pp}^{m,n} \subsetneq \V (s)$, we have the containments 
\[
\GammaAp^{m,n} \subset 
\cl \left(\V \left( \displaystyle\sum_{k=1}^n \CC[\bA,\pp] \Ap[k]^{m,1}\right) \setminus \V (s) \right)
\subset 
\cl \left(\V \left( \displaystyle\sum_{k=1}^n \CC[\bA,\pp] \Ap[k]^{m,1}\right) \setminus C_{\bA , \pp}^{m,n}\right),
\]
and~\Cref{thm:ba-set-theoretic} gives equality throughout.
On the other hand, it remains unclear whether or not~\Cref{thm:ba-set-theoretic} can be ``upgraded" to an ideal-theoretic result, since the technique used to prove the ideal-theoretic~\Cref{thm:saturation} requires working in rings where the various minors have inverses.

To summarize our results for $n\ge 1$ world point, the situation for the triangulation ideals $\Abarp^{m,n}$ and $\Abarqp^{m,n}$ is as nice as possible.
For the bundle adjustment ideals $\Ap^{m,n}$ and $\Aqp^{m,n},$~\Cref{thm:saturation} provides a characterization up to saturation by certain loci of non-generic cameras.
However, we know little about explicit generators of the bundle adjustment ideals when $n>1.$
Below, we investigate the smallest interesting cases, $\Ap^{2,2}$ and $\Ap^{2,3},$ highlighting differences from the results for $n=1$ established in~\Cref{sec:Ap}.  

\begin{example}
\label{ex:Ap-n-2-m-2}
Consider the case of $(m,n) = (2,2).$ 
The ideal $\CC[\bA, \pp] \Ap[1]^{2,1} + \CC[\bA,\pp] \Ap[2]^{2,1}$ is minimally generated by the two $2$ focals of degree $6$ and is already saturated with respect to $\langle s \rangle$. 
Therefore,   $\Ap^{2,2} = \CC[\bA,\pp] \Ap[1]^{2,1} + \CC[\bA,\pp] \Ap[2]^{2,1}$ 
as stated in \Cref{thm:saturation}. 
The equality can also be verified computationally, by checking that the right-hand side is a prime ideal of Krull dimension $34 = 11 m + 3n + m + mn$, where the right-hand side is the dimension of the affine cone $\widehat{\GammaAp^{m,n}}.$ 
The situation with Gr\"obner bases is more complicated. Using \texttt{Macaulay2}, we may verify that $\Ap^{2,2}$ has a reduced Gr\"{o}bner basis consisting of the $2$-focals together with four additional elements of degrees $8$ through $11$ for a Lex order with $\bA> \pp.$ This is not the union of any Lex Gr\"obner bases of the individual ideals, which would consist of just the two $2$-focals. 
\end{example}

\Cref{ex:Ap-n-2-m-2} shows that~\Cref{thm:Ap-GB} does not generalize for $n>1.$
Based on this example, it is tempting to conjecture that the $k$-focals might still {\em generate} $\Ap^{m,n}$ for all $m$ and $n.$
However, such a conjecture is typically false for $n>1.$
In fact, for $m=2$ cameras and $n=3$ world points, the next example shows that \emph{the $k$-focals do not even cut out $\GammaAp^{2,3}$ set-theoretically!}

\begin{example}\label{ex:Ap-n-3-m-2}
Consider, in the case of $(m,n) = (2,3),$ the ideal $\ideal$ generated by the three $2$-focals associated to a partially-symbolic arrangement $\mathbf{B} = \left(\left( \BM I & 0 \EM \right), B_2\right)$ in the ring $\CC [\mathbf{B}_2, \pp_{1 1}, \ldots , \pp_{2 3} ]$.
The ideal $\ideal$ is radical, with the prime decomposition
\[
\ideal = \ideal_{\mathbf{B}, \pp }^{2,3} \cap \langle B_2 [:, 4]\rangle .
\]
Thus, the vanishing ideal $\ideal_{\mathbf{B}, \pp }^{2,3}$ may be obtained as an ideal quotient of the focal ideal:
\[
\ideal_{\mathbf{B}, \pp }^{2,3} = \ideal : \langle B_2 [:, 4]\rangle .
\]
The prime component $\langle B_2 [:, 4] \rangle $ may be understood as defining the locus of camera pairs of the form $\mathbf{B}$, with coincident centers.
The component $\ideal_{\mathbf{B}, \pp }^{2,3} $ is generated by the $2$-focals and an additional polynomial $f(\mathbf{B}_2 , \pp )$ of degree $9$ given by
\begin{equation}\label{eq:notinAp}
f(\mathbf{B}_2 , \pp ) = \det \left(\begin{array}{c|c|c}\hspace{-.5em}
p_{2 1} \times \left( B_2 [:, 1\!:\!3] p_{1 1} \right) &  
p_{2 2} \times \left( B_2 [:, 1:3] p_{1 2} \right) &  
p_{2 3} \times \left( B_2 [:, 1:3] p_{1 3} \right) 
\end{array}\hspace{-.5em}\right),
\end{equation}
where $\times $ denotes the usual cross product.
This polynomial $f$ \emph{does not} lie in $\ideal$.
However, if we specialize to $B_2 = \bar{B}_2$ generic, then the specialization of $f$ is in the specialized focal ideal associated to $(B_1, \bar{B_2})$ by~\Cref{thm:sum-ideal}.

In the above, fixing the first camera allows us to decompose $\ideal $ more easily.
However, we note this example also allows us to prove the strict containment
\begin{equation}\label{eq:strict-2-3}
\Ap^{2,3} \supsetneq \Ap[1]^{2,1} + \Ap[2]^{2,1} + \Ap[3]^{2,1}
\end{equation}
when both cameras $A_1$ and $A_2$ are indeterminate.
To see this, define the following $4\times 4$ symbolic matrix in block form:
\[
H(A_1) = \left( \BM 
\adj A_1 [:, 1:3] & -\adj A_1 [:, 1:3] \, A_1 [:, 4]\\
0 & \det A_1 [:, 1:3]
\EM \right).
\]
This is an explicit formula for a world coordinate change $H\in \PGL_4 $ fixing the first camera to be a multiple of $\left( \BM I & 0 \EM \right)$; if we write
\[
B_1 = A_1 H(A_1) \sim  \left( \BM I & 0 \EM \right),
\quad 
B_2 = A_2 H(A_1),
\]
then $g (\bA, \pp) := f(B_2 , \pp)$ lies in the left, but not the right, ideal in~\eqref{eq:strict-2-3}.
\end{example}

In the vision literature, the polynomial $f$ in~\eqref{eq:notinAp} defines a scalar triple product constraint that was considered in several previous works, eg.~\cite{DBLP:conf/eccv/KneipSP12,ALST17}.
We emphasize that, after \emph{specializing} $\bA$ to an arrangement $\bar{\bA}$ with pairwise distinct centers, the specialized polynomial $g(\bar{\bA}, \pp)$ will lie in the triangulation ideal $\Abarp^{2,3}$ (see eg.~\cite[Appendix 2]{ALST17}.)
This \emph{does not} imply the corresponding statement for the bundle adjustment ideal $\Ap^{2,3}$.
Thus, we observe a similar phenomenon as in~\Cref{ex:Ap-4}, in which the $4$-focals were required to cut out $\GammaAp^{4,1}$; equations that are necessary to define a bundle adjustment ideal may generically specialize to redundant constraints in the associated triangulation ideal.

Returning to the triangulation ideals $\Abarp^{m,n}$ and $\Abarqp^{m,n}$, we now prove~\Cref{thm:sum-ideal} and~\Cref{thm:sum-ideal-GB}, making crucial use of the results for $n=1$ world point given in~\Cref{sec:Abarp} and~\Cref{sec:Abarqp}.

\begin{proof}[Proof of \Cref{thm:sum-ideal} and \Cref{thm:sum-ideal-GB}]
Equations~\eqref{eq:sum-Abarqp} and~\eqref{eq:sum-Abarp} follow by applying~\Cref{cor:2-3-focal} and~\Cref{thm:Abarqp-ideal}, respectively, and the fact that the vanishing ideal of the direct product of closed subvarieties in projective space is the sum of their individual vanishing ideals.
This proves~\Cref{thm:sum-ideal}.

For~\Cref{thm:sum-ideal-GB}, we may show that $G$ is a Gr\"{o}bner basis using Buchberger's S-pair criterion.
Two $g, g' \in G$ are either in the same $G_i$ or they have disjoint variable support.
In the first case, $S(g, g') \to_{G_i} 0$ implies $S(g, g') \to_{G} 0$; in the latter case, the lead terms are relatively prime, so Buchberger's first criterion~\cite[Proposition 4, pp.~106]{CLO15} implies $S(g, g') \to_G 0.$
The explicit descriptions of $G$ for suitably generic $\bar{\bA}$ follow from~\Cref{thm:AST-UGB} and ~\Cref{thm:Abarqp-gb}, respectively.
\end{proof}

Now it only remains to prove our results about the varieties $\GammaAqp^{m,n}, \GammaAp^{m,n}$ and their vanishing ideals $\Aqp^{m,n}$ and $\Ap^{m,n}.$
As in the case of the triangulation ideals treated in the previous proof, the $n=1$ results of~\Cref{sec:Ap} and~\Cref{sec:Aqp} are used in an essential way.
As a warm-up to the ideal-theoretic~\Cref{thm:saturation}, we first prove the simpler, more geometric result of~\Cref{thm:ba-set-theoretic}.

\begin{proof}[Proof of Theorem~\ref{thm:ba-set-theoretic}]
We prove the equality
\begin{equation}\label{eq:75-1}
\GammaAp^{m,n} = \cl \left( \V \left( \displaystyle\sum_{k=1}^n \CC[\bA,\pp] \Ap[k]^{m,1}\right) \setminus C_{\bA , \pp}^{m,n} \right).
\end{equation}
The proof of the second equality in the theorem statement is identical.
First, we observe that 
\begin{equation}\label{eq:75-2}
\GammaAp^{m,n} = \cl \left( \GammaAp^{m,n} \setminus C_{\bA , \pp}^{m,n} \right).
\end{equation}
This follows since $\GammaAp^{m,n}$ is an irreducible variety, and $\GammaAp^{m,n}\cap C_{\bA, \pp}^{m,n}$ is a proper subvariety of $\GammaAp^{m,n}.$
The equality
\begin{equation}\label{eq:75-3}
\cl \left( \V \left( \displaystyle\sum_{k=1}^n \CC[\bA,\pp] \Ap[k]^{m,1}\right) \setminus C_{\bA , \pp}^{m,n} \right)  =
\cl \left( \GammaAp^{m,n} \setminus C_{\bA ,  \pp}^{m,n} \right),
\end{equation}
follows by the same arguments used in the proof of~\Cref{prop:Ap-set}.
Combining~\eqref{eq:75-2} and~\eqref{eq:75-3} gives~\eqref{eq:75-1}.
\end{proof}

Finally, we prove the ideal-theoretic~\Cref{thm:saturation}.
In the proof below, we argue the result for $\Ap^{m,n}$ in~\eqref{eq:sum-Ap}.
A nearly identical argument, relying on Property \textbf{P7} of the Gr\"{o}bner basis $G_{\Aqp^{m,1}}$
in~\Cref{sec:Aqp}, allows us to show~\eqref{eq:sum-Aqp}.
The general idea is to mimic the proof of~\Cref{thm:sum-ideal-GB} by extending scalars $\CC[\bA] \hookrightarrow \CC[\bA]_s$ so that $G$ is a Gr\"{o}bner basis for the ideal it generates in the polynomial ring $\CC[\bA]_s [\pp]$.
We emphasize that $G$ need not be a Gr\"{o}bner basis in the polynomial ring $\CC[\bA, \pp],$ due to~\Cref{ex:Ap-n-2-m-2}.

\begin{proof}[Proof of Theorem~\ref{thm:saturation}]
Let $G_1, \ldots , G_n$ be the sets of $\kfour$-focals which respectively generate the ideals $\Ap[1]^{m,1}, \ldots , \Ap[n]^{m,1}$.
Set $G= G_1\cup \cdots G_n$.
We extend scalars by localizing $\CC[\bA]$ at the powers of $s$:
\[
\CC[\bA]_s = \{ f / s^k \suchthat f \in \CC [\bA], \, k \ge 0 \}.
\]
Let $\iota$ be the natural inclusion
\begin{align*}
\iota: \CC[\bA] [\pp] &\hookrightarrow \CC[\bA]_s [\pp]
\end{align*}
Note that saturating by $\langle s \rangle$ is the same as extending and contracting though the map $\iota$,
\[
\langle G \rangle : s^{\infty } =
\iota^{-1} \left( \CC[\bA]_s \, \langle G \rangle \right).
\]
We claim that
\begin{equation}\label{eq:sat-Ap}
\begin{split}
\Ap^{m,n} = \langle G \rangle : s^\infty .
\end{split}
\end{equation}
To prove this, we first argue that $\iota (G)$ forms a Gr\"{o}bner basis for the extended ideal $\CC[\bA]_s \, \langle G \rangle $ with respect to any fixed monomial order $<$.
This requires some care. 
Arguments based on Buchberger's algorithm, although valid over an arbitrary field, do not extend to arbitrary rings of scalars.
However, we may first extend the coefficient ring to the fraction field $\CC (A)$, extend $<$ to a monomial order $<'$ on $\CC (A)[\pp],$ and extend $\iota (G)$ to the set $\iota (G)' \subset \CC (A) [\pp ].$
The same argument as in~\Cref{thm:sum-ideal} shows that $\iota (G)'$ is a Gr\"{o}bner basis with respect to $<'.$
Moreover, the initial ideal $in_{<'} (\iota (G) ')$ is generated by the leading terms of $\iota (G)$ with respect to $<$, since the leading coefficients are units in either of the coefficient rings $\CC [A]_s$ or $\CC (A).$
Thus, $in_<(f)$ for any $f\in \CC[\bA]_s \, \langle G \rangle \subset \CC(\bA) \, \langle G \rangle $ is divisible by $in_< (g) \in \CC (A) [\pp ]$ for some $g\in \iota (G),$ showing that $\iota (G)$ is a Gr\"{o}bner basis for the extended ideal $\CC[\bA]_s \, \langle G \rangle  $. 

Finally, we verify that the contracted ideal $\iota^{-1} (\CC[\bA]_s \, \langle G \rangle )=\langle G \rangle : s^{\infty } $ satisfies the conditions for being the vanishing ideal of $\GammaAp^{m,n}$ given in the recognition criterion of~\Cref{prop:prop-recognition}:
\begin{enumerate}
\item The argument for the set-theoretic Condition 1 is the same as in the proof of~\Cref{thm:ba-set-theoretic}.

\item Since $s\in \frakm_A,$ it follows that $\langle G \rangle : s^\infty$ is saturated with respect to $\frakm_A.$
Moreover, \Cref{lem:grevlex} can be applied to show that the extended ideal $\CC[\bA]_s \, \langle G \rangle$ is saturated with respect to the extension of the irrelevant ideal $\CC[\bA]_s \, \frakm_\pp$.
This implies the contracted ideal $\langle G \rangle : s^\infty $ is saturated with respect to $\frakm_\pp $, since
\begin{align*}
\left( \langle G \rangle : s^{\infty } \right) : \frakm_\pp^\infty &= \left(\langle G \rangle : \frakm_\pp^{\infty } \right) : s^\infty \\
&= \iota^{-1} \left( \CC[\bA]_s \, \left( \langle  G \rangle  : \frakm_\pp^\infty \right) \right) \\
&\subset 
\iota^{-1} \left( \left( \CC[\bA]_s \, \langle  G \rangle  \right) : \left( \CC[\bA]_s \, \frakm_\pp^\infty \right) \right) \tag{\cite[Exercise 1.18]{AM69}}
\\
&= \iota^{-1} \left( \left( \CC[\bA]_s \, \langle  G \rangle  \right) \right) \tag{\Cref{lem:grevlex}}\\
&= \langle G \rangle : s^{\infty}.
\end{align*}
\item The extended ideal $\CC[\bA]_s \, \langle G \rangle $ is radical by~\Cref{prop:radical}, since it has a squarefree initial ideal.
Thus the contracted ideal $\iota^{-1} (\CC[\bA]_s \, \langle G \rangle )$ is also radical. 
\end{enumerate}
\end{proof}

\section{Open Problems and New Directions}
\label{sec:open-problems}
In this section we list some open questions about the atlas (\Cref{fig:full-diagram}) that we hope will guide future research, starting with specific questions related to those studied in this paper.

\subsection{Algebra}
\begin{enumerate}

\item In view of~\Cref{thm:AST-UGB}, it is natural to ask whether~\Cref{thm:Ap-GB} can be strengthened. 
Is the set of $\kfour$-focals is a universal Gr\"{o}bner basis of $\Ap^{m,1}$? 

\item The Gr\"{o}bner basis of $\Aqp^{m,1}$ (\Cref{sec:GB-appendix}) is complicated.
Is it possible to obtain a better understanding of its structure?
What is a universal Gr\"{o}bner basis of $\Aqp^{m,1}$? 

\item What are generators for $\Ap^{m,n}$ and $\Aqp^{m,n}$ for all $m$ and $n$?

\item What are set-theoretic and ideal-theoretic descriptions of the resectioning varieties?
\item All of our results about the specialized ideals $\Abarp^{m,n}$ and $\Abarqp^{m,n}$ make genericity assumptions implying that the camera arrangement $\bar{\bA}$ has pairwise distinct centers. 
What can be said when these assumptions are relaxed?  

\item What are the Euclidean distance degrees of the various varieties in~\Cref{fig:full-diagram}? In particular what is the Euclidean distance degree of the bundle adjustment and resectioning varieties? Note that, generally speaking the Euclidean distance degree is studied by considering nearness in {\em all} the coordinates of points on the variety. For us only nearness in $\pp$ matters. 
In general, it would be interesting to compare the nodes in the atlas via their ED degrees.
\end{enumerate}

\subsection{Geometry}
\begin{enumerate}
\item What can be said about the geometry of the multiprojective varieties in the atlas?
For instance, what are their multidegrees and their singular loci?

\item The atlas~(\Cref{fig:full-diagram}) is drawn in a vertically symmetric manner. 
This is reminiscent of a classical principle in multiview geometry known as \emph{Carlsson-Weinshall duality}~\cite{carlsson1998dual,trager2019coordinate}.
Various formulations of this principle express a duality between world points and cameras.
Here we ask what ideal-theoretic consequences of this duality are.
Can we deduce significant results about the ideals $\iideal{\qq , \pp }^{m,n}$, $\iideal{\bar{\qq}, \pp}^{m,n}$, etc, from what we already know about the ideals in~\Cref{fig:ideals}? 

\item How can the results related to $\Abarp^{m,1}$ and its Hilbert scheme in~\cite{AST13,LV20} be extended to $\Abarp^{m,n}$ and $\Abarqp^{m,n}$, beyond \Cref{thm:sum-ideal-GB}? Are there more Hilbert schemes lurking in the atlas? More precisely, when does the flat locus of the associated family give an open immersion into the Hilbert scheme as in \cite{LV20}? Does this always happen?

We note that the ideals $\Abarqp^{m,1}$ appearing in~\Cref{thm:Abarqp-gb} and~\Cref{thm:Abarqp-ideal} have the same $\ZZ^{m+1}$-graded Hilbert function and Hilbert series for any camera arrangement $\bar{\bA}$ with pairwise distinct centers.
This is immediate if $\bar{\bA}$ is ultra minor generic, since the initial ideals for a given term order are all the same. For $\bar{\bA}$ with pairwise distinct centers, this follows since the map $L_h$ defined in the proof preserves the grading. 
For $m=2$ cameras, the first few terms of this Hilbert series are 
\begin{align*}
1+4\,T_{\qq}+3\,T_{\pp_1}+3\,T_{\pp_2}+&10\,T_{\qq}^{2}+9\,T_{\qq}T_{\pp_1}+\\ &9\,T_{\qq}T_{\pp_2}+6\,T_{\pp_1}^{2}+8\,T_{\pp_1}T_{\pp_2}+6\,T_{\pp_2}^{2} + \ldots 
\end{align*}
This mirrors the corresponding result for the multiview variety given in~\cite[Theorem 3.7]{AST13}.
This could be a starting point for studying the Hilbert scheme parametrizing these ideals, parallel to the study of the ideals $\Abarp^{m,1}$ in~\cite{AST13}.

\item Degenerate camera configurations may yield very different ideals for the $\Abarqp^{m,n}$ than the specializations of $\Aqp^{m,n}$.
\Cref{ex:multiview-2-coincident} shows that the multiview variety of a tuple of concentric cameras has dimension $2$; it is really parametrizing a tuple of planar homographies. This tells us that nice behavior for $\Aqp^{m,n}$ almost ensures that we cannot capture specializations properly in degenerate cases, or, equivalently, that $\Aqp^{m,n}$ does not represent an obvious moduli problem. Is there a moduli problem that we can describe that \emph{does\/} naturally give rise to $\Aqp^{m,n}$ or some modification of it? What happens if we attempt to flatten après Raynaud and Gruson~\cite{raynaud1971criteres}? Can the result be computed or interpreted? More generally, what is the ``correct'' moduli problem for capturing degenerations?

\end{enumerate}

\subsection{Generalizations}
\begin{enumerate}
\item  When constructing the atlas studied in this paper, we assumed the simplest possible camera model, the {\em projective camera}, where the only constraint is that the camera matrix be rank 3. There are a number of specializations of this camera model which are of significant practical interest, each of which will lead to their own atlas.
We mention a few:
\begin{enumerate}
    \item Euclidean cameras, $A_i = \begin{bmatrix} R_i & t_i\end{bmatrix}$, where $R_i \in SO(3)$ and $t_i \in \RR^3$.
    \item Calibrated cameras with varying intrinsics, $A_i = K_i\begin{bmatrix} R_i & t_i\end{bmatrix}$, where $R_i \in SO(3)$ and $t_i \in \RR^3$. Here $K_i$ is the so called calibration/intrinsics matrix, which is an upper triangular matrix with positive diagonals.
    \item Calibrated cameras with common intrinsics, $A_i = K\begin{bmatrix} R_i & t_i\end{bmatrix}$, where $R_i \in SO(3)$ and $t_i \in \RR^3$. Here all cameras share the same calibration/intrinsics matrix $K$. This is the case for example when we are using a fixed focus video camera.
\end{enumerate}

The camera models described above are {\em linear}, and they do not account for the non-linear distortion caused by lenses. Accounting for them gives rise to polynomial and rational cameras and their corresponding atlases.

\item Just like atlases can be studied as we vary the kind of camera, we can also vary the scene {\em objects} being imaged. The first would be to replace world points ($\qq$) and image points ($\pp$) with lines~\cite{breiding2022line}. This is a well studied topic in 3D computer vision~\cite{taylor1995structure, bartoli2005structure}.

We can go further and consider cameras imaging a world where the objects are quadrics and their {\em images} are conics in the image plane~\cite{kahl1999affine,de1993conics}. Unlike the point and line case, where the world objects map directly to the image space objects, in this case we take a quadric surface to the boundary of its shadow in the plane.

More precisely, given a camera $\PP^3\dashrightarrow\PP^2$, a general quadric $Q$ in $\PP^3$ has the property that the camera defines a finite morphism $Q\to\PP^2$ whose set of critical values (the branch curve) is a conic curve $C\subset\PP^2$. This defines a rational map of linear spaces $\PP^{9}\dashrightarrow\PP^{5}$ from the space of quadrics in $\PP^3$ to the space of conics in $\PP^2$.  Similarly, a tuple of cameras defines a rational map in this way from $\PP^9\to\Hilb_{(\PP^2)^m}$.
When $m=1$, the relevant component of the Hilbert scheme of $\PP^2$ is the linear space of conics.

\item The atlas can also be generalized by replacing the matrix-vector multiplication $A q $ with arbitrary matrix-matrix multiplication, i.e. 
    \begin{equation}
    C_{ij}  \sim A_{i} B_j        
    \end{equation}
    where $A_{i} \in \PP^{\alpha \times \beta }, B_{j} \in \PP^{\beta  \times \gamma}$ and $C_{ij} \in \PP^{\alpha \times \gamma}$.  This seems like a very general object and worthy of study on its own.    
\end{enumerate}

\bibliographystyle{amsplain}
\bibliography{refs.bib}
\newpage 
\appendix 
\section{Dimension counts}
\label{sec:dimension-counts}
\begin{proposition}
\label{prop:dim-counts}
Below, $\bar{\bA} \in (\PP^{11})^m, \bar{\qq} \in (\PP^3)^n, \bar{\pp} \in \Gamma^{m,n}_{\pp}$ 
are generic whenever they appear.
\begin{align}
\dim \Gamma^{m,n}_{\bA,\qq,\pp} &= 3n + 11m \label{eq:dim-Aqp}\\
\dim \Gamma^{m,n}_{\bA,\pp} &= \min (2mn + 11m, 3n+11m) \label{eq:dim-Ap}\\
\dim \Gamma^{m,n}_{\qq, \pp} &= \min(3n+2mn, 3n + 11m) \label{eq:dim-qp}\\
\dim \Gamma^{m,n}_{\pp} &= \min(2mn, 11m + \max(3n - 15, 0)) \label{eq:dim-p}\\
\dim \Gamma^{m,n}_{\bar{\bA}, \qq, \pp} &= \dim \Gamma^{m,n}_{\bA, \qq, \pp} - \dim \Gamma^{m,n}_\bA  = 3n \label{eq:dim-Abarqp}
\\
\dim \Gamma^{m,n}_{\bA, \qq, \bar{\pp}} &= \dim \Gamma^{m,n}_{\bA, \qq, \pp} - \dim \Gamma^{m,n}_{\pp} \label{eq:dim-Aqpbar}\\
\dim \Gamma^{m,n}_{\bA, \bar{\qq}, \pp} &= \dim \Gamma^{m,n}_{\bA, \qq , \pp} - \dim \Gamma^{m,n}_{\qq} = 11m \label{eq:dim-Aqbarp}\\
\dim \Gamma^{m,n}_{\bar{\bA}, \pp} &= \dim \Gamma^{m,n}_{\bA, \pp} - \dim \Gamma^{m,n}_{\bA} = \min (2mn, 3n)\label{eq:dim-Abarp}\\
\dim \Gamma^{m,n}_{\bA, \bar{\pp}} &= \dim \Gamma^{m,n}_{A, \pp} - \dim \Gamma^{m,n}_{\pp}   \label{eq:dim-Apbar}\\
\dim \Gamma^{m,n}_{\qq, \bar{\pp}} &= \dim \Gamma^{m,n}_{\qq , \pp } - \dim \Gamma^{m,n}_{\pp} \label{eq:dim-qpbar}\\
\dim \Gamma^{m,n}_{\bar{\qq} , \pp} &= \dim \Gamma^{m,n}_{\qq, \pp} - \dim \Gamma^{m,n}_{\qq} \label{eq:dim-qbarp}
\end{align}
\end{proposition}

\begin{remark}
For $m$ and $n$ sufficiently large, the formulas above involving $\min $ and $\max $ expressions can be simplified as follows:
\begin{align}
\dim \Gamma^{m,n}_{\bA,\pp} = \dim \Gamma^{m,n}_{\qq, \pp} &=  3n+11m \\
\dim \Gamma^{m,n}_{\pp} &= 3n + 11m - 15\\
\dim \Gamma^{m,n}_{\bar{\bA}, \qq, \pp} = \dim \Gamma^{m,n}_{\bar{\bA}, \pp} &=   3n 
\\
\dim \Gamma^{m,n}_{\bA, \qq, \bar{\pp}} = \dim \Gamma^{m,n}_{\bA, \bar{\pp}} = \dim \Gamma^{m,n}_{\qq, \bar{\pp}} &= 15 \\
\dim \Gamma^{m,n}_{\bA, \bar{\qq}, \pp} = \dim \Gamma^{m,n}_{\bar{\qq} , \pp}  &=  11m
\end{align}
\end{remark}
\begin{proof}
$ $\\[1em]
\eqref{eq:dim-Aqp}: This follows at once from the birational equivalence
\begin{align*}
\pi_\pp : \GammaAqp^{m,n}&\to (\PP^{11})^m \times (\PP^3)^n \\
(\bar{\bA}, \bar{\qq}, \bar{\pp}) &\mapsto (\bar{\bA}, \bar{\qq}) .
\end{align*}
\eqref{eq:dim-Ap}: Consider the projection
\begin{align*}
\pi_\qq : \GammaAqp^{m,n}&\to (\PP^{11})^m \times (\PP^2)^{mn}\\
(\bar{\bA}, \bar{\qq}, \bar{\pp}) &\mapsto (\bar{\bA}, \bar{\pp}).
\end{align*}
A fiber $\pi_\qq^{-1} (\bar{\bA}, \bar{\pp})$ can be identified with the projective linear space of all $q$ satisfying $A_i q_j \sim p_{i j}$.
Equivalently, each of the $2mn$ matrices $\left(\BM A_i q_j & p_{i j}\EM\right)$ is of rank one. 
For generic $(\bar{\bA}, \bar{\pp}) \in \Gamma^{m,n}_{\bA,\pp },$ the $2\times 2$ minors impose $2$ linear conditions on $\bar{\qq},$ so that $\pi_\qq^{-1} (\bar{\bA}, \bar{\pp})$ is a projective linear space of dimension $\max (3n-2mn, 0).$
Hence, using~\eqref{eq:dim-Aqp} and the fiber-dimension theorem, 
\begin{align*}
\dim \Gamma^{m,n}_{\bA,\pp} &= \dim \Gamma^{m,n}_{\bA,\qq, \pp} - \max (3n-2mn, 0)\\
&= \min (2mn + 11m, 3n + 11m).
\end{align*}
\eqref{eq:dim-qp}: Similar to~\eqref{eq:dim-Ap}, consider 
\begin{align*}
\pi_\bA : \Gamma^{m,n}_{\bA,\qq. \pp} &\to (\PP^{11})^m \times 
(\PP^2)^{mn} \\
(\bar{\bA}, \bar{\qq}, \bar{\pp}) &\mapsto (\bar{\qq}, \bar{\pp}).
\end{align*}
For generic $(\bar{\qq}, \bar{\pp}) \in \Gamma^{m,n}_{\qq, \pp},$ the fibers $\pi_\bA^{-1} (\bar{\qq}, \bar{\pp})$ are projective linear spaces of dimension $\max(11m - 2mn, 0)$ which are defined by the $\bA$-linear $2\times 2$ minors of $\left(\BM A_i q_j & p_{i j}\EM\right).$ 
Hence
\begin{align*}
\dim \Gamma^{m,n}_{\qq,\pp} &= \dim \Gamma^{m,n}_{\bA,\qq, \pp} - \max (11m-2mn, 0)\\
&= \min (3n + 2mn, 3n + 11m).
\end{align*}
\eqref{eq:dim-p}:
For generic $q_1, \ldots , q_5 \in \PP^3,$
define 
$\sigma = \{ \sigma_1, \ldots , \sigma_4 \} \subset [5],$
\[[\sigma_1 \, \sigma_2 \, \sigma_3 \, \sigma_4]^{-1} := 
\det \left(\BM q_{\sigma_1} & q_{\sigma_2} & q_{\sigma_3}  & q_{\sigma_4}\EM\right)^{-1}, 
\]
and consider the \emph{projective change of basis matrix}
\begin{equation}
\label{eq:cremona-P3}
\begin{split}
C_{q_1, \ldots , q_5} = 
&
\, \diag \left(
[5 2 3 4]^{-1}, \, [1 5 3 4]^{-1}, \,
[1 2 5 4]^{-1}, \,
[1 2 3 5]^{-1}
\right)
\cdot \left(\BM q_1 & q_2 & q_3 & q_4 \EM\right)^{-1}
\end{split}
\end{equation}
where $\bullet^{-1}$ denotes matrix inversion.
When defined, the projective transformation defined by $C_{q_1, \ldots , q_5}$ maps $q_1 \ldots q_5$ onto the \emph{standard projective basis}:
\begin{align*}
C_{q_1, \ldots , q_5} \cdot q_1 &\sim e_1, \\ 
C_{q_1, \ldots , q_5} \cdot q_2 &\sim e_2, \\ 
C_{q_1, \ldots , q_5} \cdot q_3 &\sim e_3, \\
C_{q_1, \ldots , q_5} \cdot q_4 &\sim e_4, \\
C_{q_1, \ldots , q_5} \cdot q_5 &\sim e_1 + e_2 + e_3 + e_4.
\end{align*}
Consider the projection
\begin{align*}
\pi_{\bA, \qq} : \GammaAqp^{m,n} &\to \Gamma^{m,n}_\pp \\
(\bar{\bA}, \bar{\qq}, \bar{\pp}) &\mapsto \bar{\pp} .
\end{align*}
We observe a version of \emph{projective ambiguity}~\cite[p~265]{HZ04}, stating that the fibers of $\pi_{\bA, \qq}$ are invariant under the action of $ \PGL_4$ described in~\Cref{sec:genericity}.
Suppose first that $n<6$.
We need to show $\dim \Gamma^{m,n}_\pp = 2mn.$
Let $\bar{\pp} \in \Gamma^{m,n}_\pp$
be generic and suppose $(\bar{\bA}, \bar{\qq}, \bar{\pp}) \in \pi_{\bA, \qq}^{-1} (\bar{\pp}).$
Then for generic $\bar{\qq} ' \in (\PP^3)^n,$ we may find $(\bar{\bA}', \bar{\qq} ' , \bar{\pp} ) \in \pi_{\bA, \qq}^{-1} (\bar{\pp})$ by projective change of basis $H = C_{\tilde{\qq} '}^{-1} C_{\tilde{\qq}},$ where $\tilde{\qq}$ and $\tilde{\qq}' $ extend $\bar{\qq}$ and $\bar{\qq} ' $ to projective bases when $n<5.$
In other words, the generic fiber of $\Gamma^{m,n}_{\qq, \pp} \to \Gamma^{m,n}_\pp$ has dimension $3n.$
Applying~\eqref{eq:dim-Aqp},
\begin{align*}
\dim \Gamma^{m,n}_\pp &= \dim \Gamma^{m,n}_{\qq, \pp} - 3n = (3n + 2mn) - 3n = 2mn.
\end{align*}
Now suppose $n\ge 6$.
For $m=1$ camera,~\eqref{eq:dim-p} asserts that $\dim \Gamma^{m,n}_\pp = 2 n ,$ which follows since there are no constraints on image points.
Otherwise, observe that the quantity
\begin{align*}
\label{eq:codim-p}
2mn - (11m + 3n - 15) &= (2m-3) n - 11m + 15 = (2n - 11) m + 15 - 3n
\end{align*}
is increasing in $n$ for fixed $m\ge 2$ and increasing in $m$ for fixed $n\ge 6.$
Moreover, this quantity equals zero precisely in the \emph{minimal cases} $(m,n)=(2,7), \, (3, 6).$
Thus,~\eqref{eq:dim-p} asserts that $\dim \Gamma^{m,n}_{\pp } = 11m + 3n - 15$ whenever either $m\ge 2$ and $n\ge 7$ or $m\ge 3$ and $n\ge 6.$
This leaves one exceptional case for $n\ge 6,$ which is $(m,n) = (2,6)$; here, to show that $\dim \Gamma^{m,n}_{\pp} = 2mn = 24$, it suffices to verify that the Jacobian of $\pi_{\bA, \qq}$ evaluated at some point in local coordinates has rank $24.$ 
The same Jacobian check gives us $\dim \Gamma^{m,n}_{\pp } = 11m + 3n - 15$ for the two minimal cases; equivalently, $\dim \pi_{\bA, \qq}^{-1} (\bar{\pp})= 15$ for generic $\bar{\pp} \in \Gamma^{m,n}_\pp .$
Finally, if either $m\ge 2$ and $n\ge 8$ or $m\ge 3$ and $n\ge 7,$ note that the fiber $\pi_{\bA,\qq}^{-1} (\bar{\pp}) $ for generic $\bar{\pp} \in \Gamma^{m,n}_\pp$ is nonempty, and thus has dimension at least $15$ by projective ambiguity.
Since $\pi_{\bA,\qq}^{-1} (\bar{\pp}) $ projects onto a fiber for one of the minimal cases, we also have $\dim \pi_{\bA, \qq}^{-1} (\bar{\pp}) \le 15.$
Thus 
\begin{align*}
\dim \Gamma^{m,n}_\pp &= \dim \Gamma^{m,n}_{\bA, \qq, \pp} - \dim \dim \pi_{\bA, \qq}^{-1} (\bar{\pp}) = 11m + 3n - 15.
\end{align*}
\eqref{eq:dim-Abarqp}--\eqref{eq:dim-qbarp} In all cases, $\Gamma^{m,n}_{\bar{X}, Y}$ is the generic fiber of $\Gamma^{m,n}_{X ,Y} \to \Gamma^{m,n}_{X},$ so these formulas follow from the fiber dimension theorem and~\eqref{eq:dim-Aqp}--\eqref{eq:dim-p}.
\end{proof}

\section{Miscellaneous Proofs}
\label{sec:proofs appendix}

\subsection{Generic Cameras}
\begin{proposition} \label{prop:easy directions}
If a camera arrangement $\bar{\bA} = (\bar{A}_1,\ldots \bar{A}_m)$ is ultra minor generic 
then it is minor generic, and if $\bar{\bA}$ is minor generic, then it has pairwise distinct centers.
\end{proposition}
\begin{proof}
If $\bar{\bA}$ is ultra minor generic, then all $k \times k$ minors of 
$\stacked{\bar{A}_1}{\bar{A}_m}$ are nonzero for any $k \in [4]$. In particular, 
all $4 \times 4$ minors are nonzero and $\bar{\bA}$ is minor generic. If 
$\bar{\bA}$ is minor generic, then for any $1 \leq i < j \leq m$, 
the $4\times 6$ matrix $\left( \begin{array}{c|c}\hspace{-.4em}{\bar{A}_i}^\top&{\bar{A}_j}^\top\end{array}\hspace{-.4em}\right)$ has rank $4.$
This implies that $\bar{A}_i$ and $\bar{A}_j$ have district centers.
\end{proof}

\begin{theorem} 
\begin{enumerate}
    \item A camera arrangement $\bar{\bA}$ has pairwise distinct centers if and only if it is equivalent to a minor generic camera arrangement under the group action  \eqref{eq:small group-action}.
    \item A camera arrangement $\bar{\bA}$ is minor generic if and only if it is equivalent to an ultra minor generic arrangement under the group action  \eqref{eq:PGL4  group-action}.
    \item A camera arrangement $\bar{\bA}$ has pairwise distinct centers if and only if it is equivalent to an ultra minor generic camera arrangement under the group action \eqref{eq:big group-action}. 
\end{enumerate}
\end{theorem}
\begin{proof}
\begin{enumerate}
    \item This statement was proved in ~\cite[Lemma 3.6]{APT21}. 
    
    \item We already saw in \Cref{prop:easy directions} that ultra minor genericity implies minor genericity. For the other direction, fix a minor generic arrangement $(\bar{A}_1,\ldots \bar{A}_m)$. 
Let $\sigma \in \binom{[4]}{k}$ and $\tau \in \binom{[3m]}{k}$ be subsets indexing the rows and columns of some $k\times k$ minor of 
$\stacked{\bar{A}_1}{\bar{A}_m}$.
Using the Cauchy-Binet theorem,
\begin{equation}
    \label{eq:cauchy-binet-H}
    \begin{split}
    \det \left( \stacked{(\bar{A}_1 H)}{(\bar{A}_m H)} [\sigma, \tau] \right)
    =\\[0.618em]
        \displaystyle\sum_{\upsilon \in \binom{[4]}{k}} 
        \det \left( H^\top[\sigma, \upsilon] \right) \cdot 
        \det \left(\stacked{\bar{A}_1}{\bar{A}_m}[\upsilon,\tau] \right).
      \end{split}
\end{equation}
The minors $\det(\stacked{\bar{A}_1}{\bar{A}_m}[\upsilon,\tau])$ which occur in this sum range over all $k\times k$ minors of the $4 \times k$ matrix $\stacked{\bar{A}_1}{\bar{A}_m}[[4], \tau].$
This $4\times k$ matrix has full rank $k$ since if it did not, we could add $4-k$ additional columns from $\stacked{\bar{A}_1}{\bar{A}_m}$ to get a rank-deficient $4\times 4$ matrix, contradicting our assumption that 
$\stacked{\bar{A}_1}{\bar{A}_m}$ is minor generic.
Thus, $\det \stacked{\bar{A}_1}{\bar{A}_m}[\upsilon,
\tau] \ne 0$ for some $\upsilon \in \binom{[4]}{k}.$
Hence the expressions in \eqref{eq:cauchy-binet-H} are not zero, and setting ~\eqref{eq:cauchy-binet-H} to $0$ we obtain a  hypersurface in $\GL_4$.
Any choice of $H$ lying outside the union of the finitely many hypersurfaces, obtained by varying 
over all $k, \sigma, \tau$ yields an arrangement $(A_1 H, \ldots , A_m H)$ satisfying the conclusion.
    
\item This statement follows from the first two.
    
\end{enumerate}
\end{proof}

\subsection{Proof of~\Cref{prop:radical}, part 1}

\begin{proof}
Let $I = \langle g_1, \ldots , g_s \rangle .$
By definition, $g_1, \ldots ,g_s$ forming a Gr\"{o}bner basis with respect to $<$ means that \[
in_< (I) = \langle in_< (g_1), \ldots , in_< (g_s) \rangle .\]
First, we verify that this monomial ideal in $R[x_1,\ldots , x_k]$ is radical. 
To see this, let $in_<(g_s) = x_{i_1} \cdots x_{i_l}$ and note that 
\[
in_< (I) = \displaystyle\bigcap_{j=1}^l \langle in_< (g_1) , \ldots , in_< (g_{s-1}), x_{i_j} \rangle .
\]
Iterating this argument, we obtain $in_< (I)$ as an intersection of prime ideals generated by subsets of the variables.

Now, to show that $I = \langle g_1, \ldots , g_s \rangle $ is radical, suppose that $f \in  \sqrt{I},$ so that $f^n \in I$ for some positive integer $n.$ We need to argue that $f \in I$.
We have 
\[
in_<(f)^n = in_< (f^n) \in in_< (I)
\quad 
\Rightarrow
\quad 
in_< (f) \in in_< (I),
\]
where the first equality of leading terms uses the fact that $R$ is a domain and the implication holds since $in_<(I)$ is radical.
Thus there exists $f_0 \in I \subseteq \sqrt{I}$ such that $in_< (f_0) = in_< (f).$
Now $f-f_0 \in \sqrt{I}$ is an element whose leading term is strictly smaller than $in_<(f).$
Replacing $f$ with $f-f_0$ and iterating the argument, we obtain 
$f_0, \ldots , f_l \in I$ such that $f = f_0 + \cdots + f_l \in I$. 
\end{proof}

\subsection{Proof of the recognition criterion: \Cref{prop:prop-recognition}}

\begin{proof}
A point $x \in \PP^{n_1} \times \cdots \times \PP^{n_k}$ may be represented in homogeneous coordinates by a point $\hat{x}$ in the affine space $\CC^{n_1+1} \times \cdots \times \CC^{n_k+1}.$  
Consider the \emph{affine cone}
\[
\hat{X} = \cl \{ \hat{x} \in \CC^{n_1+1} \times \cdots \times \CC^{n_k+1} \suchthat x \in X \}.
\]
This is an \emph{affine} variety whose vanishing ideal is precisely the vanishing ideal of $X.$

Suppose that Conditions 1--3 are satisfied; we must show that $\langle f_1, \ldots , f_s \rangle $ is the vanishing ideal of $X,$ or equivalently that of $\hat{X}.$
Condition 3 and the Nullstellensatz~\cite[Ch.~4, \S 2]{CLO15} imply that $\langle f_1, \ldots , f_s \rangle $ is the vanishing ideal of the affine variety
\[
\Vaff (f_1, \ldots , f_s) = \{ \hat{x} \in \CC^{n_1+1} \times \cdots \times \CC^{n_k+1} \suchthat f_1(\hat{x})=\cdots = f_s(\hat{x}) \}.
\]
Moreover, Conditions 2 and 3 together with standard properties of ideal quotients and saturation~\cite[Ch.~4, \S 4]{CLO15} imply $\langle f_1, \ldots , f_s \rangle $ is the vanishing ideal of
\[
\cl \left( \Vaff (f_1, \ldots , f_s) \setminus \Vaff ( \frakm_{\xx_1} \cap \cdots \cap \frakm_{\xx_k})\right).
\]
Since affine varieties are uniquely determined by their vanishing ideals, it is now enough to observe the following equality, which holds whenever Condition 1 is satisfied:
\[
\hat{X} = \cl \left( \Vaff (f_1, \ldots , f_s) \setminus \Vaff ( \frakm_{\xx_1} \cap \cdots \cap \frakm_{\xx_k}) \right).
\]
Conversely, we verify Conditions 1--3 when $\langle f_1, \ldots , f_s \rangle $ is the vanishing ideal of $X$:
\begin{enumerate}
    \item $ X \subset \V (f_1, \ldots , f_s)$ since each $f_i$ vanishes on all points of $X$.
On the other hand, we have $X = \V (g_1, \ldots , g_s)$ for certain homogeneous polynomials $g_1, \ldots , g_s,$ all of which must be contained in $\langle f_1, \ldots , f_s \rangle $.
If $f_1,\ldots , f_s$ vanish at a point, so must $g_1, \ldots, g_s$, and thus $\V (f_1, \ldots , f_s) \subset X.$
\item Let $f \in \langle f_1, \ldots , f_s \rangle : (\frakm_{\xx_1} \cap \cdots \cap \frakm_{\xx_k})^\infty.$
To show $f \in \langle f_1, \ldots , f_s \rangle ,$ it is enough to show that each of the homogeneous components of $f$ vanish on $X,$ so suppose further that $f$ is homogeneous. 
If $X$ is empty, then $f$ vanishing on $X$ holds vacuously.
Otherwise, for any point in $X$ there exists some monomial of the form
\[
m(x) = x_{1, i_1} \, \cdots \, x_{k, i_k}  \in \frakm_{\xx_1} \cap \cdots \cap \frakm_{\xx_k}
\]
which does not vanish at that point.
Since $m(x)^n f$ is in the vanishing ideal for some $n\ge 1,$ we see that that $f$ must vanish at this point.
    \item If $f^n (x) =0$ for some $n\ge 1$ and all $x\in X$, then $f(x) =0$ for all $x\in X.$
\end{enumerate}
\end{proof}

\section{Generators of $G_{\minorideal}$}
\label{sec:GB-appendix}
The Gr\"{o}bner basis $G_{M_{\bA,\qq,\pp}^{m,1}}$ of~\Cref{prop:MAqp-GB} contains elements in degree $3,4,5,6,7,8,9.$ Here, for completeness, we give explicit formulas for all of them.
\\\\
\textbf{Degree 3:} $3m$ elements---for $1 \le i \le m ,$ $1\le  i_1 < i_2 \le 3, \, \det \Big( \left(\BM A_i q & p_i \EM\right) [ \{i_1, i_2 \} , :] \Big).$
\\\\
\textbf{Degree 4:} $m$ elements---for $1 \le i \le m , \, \det \left(\BM A_i[:,1] & A_i q & p_i \EM\right).$
\\\\
\textbf{Degree 5:} $9\, {m \choose 2}$ elements---for $1 \le i < j \le  m ,$ $1 \le i_1 < i_2 \le 3,$ $1\le  j_1 < j_2 \le 3,$
    \begin{flalign*}
    \det \Big( \left(\BM A_i[:,1] & p_i\EM\right) [\{i_1, i_2\} , :] \Big) \cdot 
    \det \Big( \left(\BM A_j q  & p_j\EM\right) [ \{j_1, j_2 \} , : ] \Big)  & -\\ 
    \, \det \Big( \left(\BM A_j[:,1] & p_j\EM\right) [\{j_1, j_2\}, :] \Big) \cdot 
    \det \Big( \left(\BM A_i q  & p_i\EM\right) [\{i_1, i_2\}, :] \Big)  & .
    \end{flalign*}
\textbf{Degree 6:} $6\, {m \choose 2}$ elements---for $1 \le k_1 < k_2 \le 3,$ $1\le i , j \le m,$ $i\ne j,$
    \begin{flalign*}
\sum_{l=1}^3 (-1)  \cdot 
  &\det \Big(\left(\BM A_j q & p_j \EM\right) [\{ 1,2,3 \} \setminus l , :] \Big) \, \cdot 
 \Bigg( p_i[k_1] \det \left(\BM A_i [k_2, 1:2] \\ 
                                  A_j [l, 1:2]
                    \EM\right)                 
 - p_i[k_2] \det \left(\BM A_i [k_1, 1:2] \\ 
                               A_j [l,1:2]
                    \EM\right)
\Bigg) \\
+ &\det \left(\BM A_j[:,1] & A_j[:,2] & p_j \EM\right) \cdot \det \Big(\left(\BM A_i q & p_i\EM\right) [ \{k_1, k_2\}, :] \Big) .
\end{flalign*}
\textbf{Degree 7:} $\binom{m}{2}$ elements of the form $q_4$ times a $2$-focal, plus an additional $27 \binom{m}{3}$ elements---for $1\le i <j < k \le m,$ $1\le i_1 < i_2\le 3,$ $1\le j_1 < j_2 \le 3,$ $1\le k_1 < k_2 \le 3,$
    \begin{flalign*}
    \det \Big( \left(\BM A_i q & p_i \EM\right)[\{i_1, i_2\}, :] \Big) \cdot 
    \Bigg(
    &p_j[j_1] p_k[k_1] \cdot \det \left(\BM A_j [j_2, 1:2] \\ A_k [k_2, 1:2] \EM\right) - p_j[j_1] p_k[k_2] \cdot \det \left(\BM A_j [j_2, 1:2]\\ A_k [k_1, 1:2] \EM\right) -\\
    &p_j[j_2] p_k[k_1] \cdot \det \left(\BM A_j [j_1, 1:2] \\ A_k [k_2, 1:2] \EM\right) + p_j[j_2] p_k[k_2] \cdot \det \left(\BM A_j [j_1, 1:2]\\ A_k [k_1, 1:2] \EM\right)
    \Bigg)-\\
    \det \Big( \left(\BM A_j q & p_j \EM\right) [\{j_1, j_2\}, :] \Big) \cdot 
    \Bigg(
    &p_i[i_1] p_k[k_1] \cdot \det \left(\BM A_i [i_2, 1:2] \\ A_k [k_2, 1:2] \EM\right) - p_i[i_1] p_k[k_2] \cdot \det \left(\BM A_i [i_2, 1:2] \\ A_k [k_2, 1:2] \EM\right) -\\ 
    &p_i[i_2] p_k[k_1] \cdot \det \left(\BM A_i [i_2, 1:2] \\ A_k [k_2, 1:2] \EM\right) + p_i[i_2] p_k[k_2] \cdot \det \left(\BM A_i [i_1, 1:2] \\ A_k [k_1, 1:2] \EM\right) 
    \Bigg)+ \\
    \det \Big(\left(\BM A_k q & p_k \EM\right) [\{k_1, k_2\}, :] \Big) \cdot 
    \Bigg(
      &p_i[i_1] p_j[j_1] \cdot \det \left(\BM A_i [i_2, 1:2] \\ A_j [j_2, 1:2] \EM\right) - p_i[i_1] p_j[j_2] \cdot \det \left(\BM A_i [i_2, 1:2] \\ A_j [j_1, 1:2] \EM\right) -\\ 
    &p_i[i_2] p_j[j_1] \cdot \det \left(\BM A_i [i_1, 1:2] \\ A_j [j_2, 1:2] \EM\right) + p_i[i_2] p_j[j_2] \cdot \det \left(\BM A_i [i_1, 1:2]\\ A_j [j_1, 1:2] \EM\right) 
    \Bigg)  
    \end{flalign*}
\textbf{Degrees 8 \& 9:} $q_4$ times $27 \binom{m}{3}$ $3$-focals and $q_4$ times $81 \binom{m}{4}$ $4$-focals, respectively. 
\end{document}